\documentclass[a4paper,12pt,reqno]{amsart}

\usepackage{amsmath,amssymb,amsthm}
\usepackage{latexsym}
\usepackage{color}
\usepackage{graphicx}
\usepackage{mathrsfs}
\usepackage{enumerate}
\usepackage{enumitem}
\usepackage[abbrev]{amsrefs}
\usepackage[T1]{fontenc}
\usepackage{bbm}
\usepackage{mathtools}
\mathtoolsset{showonlyrefs=true}

\usepackage{tikz}
\usetikzlibrary{arrows.meta}
\usepackage{float}

\usepackage[colorlinks,
            linkcolor=blue,      
           anchorcolor=blue,  
          citecolor=red,       
          ]{hyperref}

\allowdisplaybreaks[4]

\theoremstyle{plain}
\newtheorem{thm}{Theorem}[section]
\newtheorem{lemm}[thm]{Lemma}
\newtheorem{prop}[thm]{Proposition}

\theoremstyle{definition}
\newtheorem{df}[thm]{Definition}

\makeatletter

\@addtoreset{equation}{section}
\makeatother

\renewcommand{\div}{\operatorname{div}}
\newcommand{\dB}{\dot{B}}

\newcommand{\supp}{\operatorname{supp}}

\renewcommand{\leq}{\leqslant}
\renewcommand{\geq}{\geqslant}

\newcommand{\n}[1]{{\left\|#1\right\|}}

\newcommand{\Id}{{\rm Id}}

\newcommand{\R}{\mathbb{R}}
\newcommand{\T}{\mathbb{T}}
\newcommand{\N}{\mathbb{N}}
\newcommand{\Z}{\mathbb{Z}}

\newcommand{\lp}[1]{\left[#1\right]}
\newcommand{\Mp}[1]{\left\{#1\right\}}
\renewcommand{\sp}[1]{\left(#1\right)}
\newcommand{\abso}[1]{\left|#1\right|}

\newcommand{\bb}{\mathbbm}

\begin{document}
\title[Non-uniqueness for the $L^2$ steady Navier--Stokes flow on the plane]
{Non-unique $L^2$ solutions to the stationary Navier--Stokes equations on the whole plane}
\author{Mikihiro Fujii}
\address{Graduate School of Science, Nagoya City University, Nagoya, 467-8501, Japan}
\email{fujii.mikihiro@nsc.nagoya-cu.ac.jp}
\begin{abstract}
We prove that, for a suitable small external force in $\dot H^{-2}(\mathbb R^2)$, the stationary Navier--Stokes equations on $\mathbb R^2$ admit multiple solutions in $L^2(\mathbb R^2)$.
\end{abstract}
\maketitle


\section{Introduction}
Let us consider the stationary problem for the incompressible Navier--Stokes equations on the whole space $\R^n$ with $n \geq 2$:
\begin{align}\label{eq:sNS}
    -\Delta u + \mathbb P (u \cdot \nabla)u =  f,
    \qquad 
    \div u = 0,
\end{align}
where $u:\R^n \to \R^n$ is the unknown velocity field of the fluid, while $f:\R^n \to \R^n$ is the given external force satisfying $\div f = 0$.
We denote by $\mathbb P$ the Helmholtz projection defined by $\mathbb{P}=\mathscr{F}^{-1}\widehat{\mathbb P}(\xi)\mathscr{F}$, where $\widehat{\mathbb P}(\xi)=\Id - \xi \otimes \xi/|\xi|^2$.
The aim of this paper is to provide a negative answer to the uniqueness problem of small solutions to \eqref{eq:sNS} with $n=2$ in the critical space $L^2(\R^2)$ with some small external force in the corresponding space $\dot H^{-2}(\R^2)$.

Before stating our main result precisely, we recall known results related to our work.
For the case of $n \geq 3$, Chen \cite{Che-93} proved the existence and uniqueness of small solutions $u$ in the scaling critical space $L^n(\R^n)$ for given small external forces\footnote{In the original article \cite{Che-93}, the assumption is slightly different and he considered the external forces of the divergence form $f=\div F$ with some $F \in L^{n/2}(\R^n)$, whereas we may use a slightly weaker assumption $f \in \dot W^{-2,n}(\R^n)$ by the same proof.} $f \in \dot W^{-2,n}(\R^n)$.
For small solutions to \eqref{eq:sNS} in scaling critical Besov spaces $\dB_{p,q}^{n/p-1}(\R^n)$, the results from
Kaneko--Kozono--Shimizu \cite{Kan-Koz-Shi-19} and Chen \cite{Che-93} imply that uniqueness holds
for the case of $1 \leq p < n$ and $1 \leq q \leq \infty$ or $p=n$ and $1 \leq q \leq 2$.
The author \cite{Fujii-26} recently proved that uniqueness fails for all the other cases of $(p,q)$.
Note that exactly the same results holds for the critical Lizorkin--Triebel spaces. See \cites{Fuj-24,Tsu-19-ARMA,Tsu-19-DIE} for the related topics.
Therefore, $L^n(\R^n)$ is a threshold space between uniqueness and non-uniqueness.

For the two-dimensional case $n=2$, the situation is more difficult than that in higher dimensions.
A fundamental obstruction in unbounded two-dimensional domains is the Stokes paradox: because the fundamental solution of the two-dimensional Stokes system grows logarithmically, the corresponding stationary Stokes problem with nonzero total force does not, in general, admit a solution decaying at spatial infinity; see \cite[Chapter XII]{Galdi-11}.
Consequently, the construction of stationary Navier--Stokes flows in unbounded planar domains is particularly delicate except for perturbative regimes around specially chosen stationary background flows.
To construct the stationary solutions on $\R^2$, one should impose some special conditions on the external force.
Yamazaki \cite{Yam-09} proved the existence and uniqueness of small solutions in a function space close to $L^{2,\infty}(\R^2)$ under suitable antisymmetry assumptions on the external force.
Guillod--Korobkov--Ren \cite{GKR-23-CMP} proved the existence of solutions for compactly supported forces.
However, it has seemed to be a difficult problem to establish the solvability in the $L^2$ framework, that is the two-dimensional version of \cite{Che-93}, even if we impose any additional condition for external forces.
There are some results that focus on this problem from the negative aspect for uniqueness on the torus.
Lemari\'e–Rieusset \cite{Rieusset-25-JFA} constructed a nontrivial solution $u \in H^{-1}(\T^2)$ to \eqref{eq:sNS} with $f \equiv 0$.
In \cite{ABGN25}, the regularity framework was improved to
$\dot{H}^{-\varepsilon}(\T^2) \cap L^{2-\varepsilon}(\T^2)$ 
for arbitrary $0<\varepsilon < 1$.
Cheskidov--Hou \cite{Che-Hou-26} constructed a nontrivial unforced stationary Navier--Stokes flow in $B_{p,q}^s(\T^2)$ for every $s<0$ and $1 \leq p,q \leq \infty$.
However, none of these non-uniqueness results resolves the open uniqueness question for $L^2$ solutions, as the convex integration they used requires strong singularity for solutions.
On the other hand, uniqueness of two-dimensional small solutions holds in $\dB_{2,1}^0(\R^2)$, which is slightly narrower\footnote{We notice that $L^2(\R^2) = \dot B_{2,2}^0(\R^2)$ and $\dB_{2,q}^s(\R^2) \hookrightarrow \dB_{2,r}^s(\R^2)$ for $1 \leq q \leq r \leq \infty$ and $s \in \R$.}
than $L^2(\R^2)$.
Indeed, if $u,v$ are two solutions in $\dB_{2,1}^0(\R^2)$ with the same external force, then we see by the paraproduct estimate $\| fg \|_{\dB_{2,\infty}^{-1}} \leq C\| f \|_{\dB_{2,1}^0} \| g \|_{\dB_{2,\infty}^0}$ that
\begin{align}
    \n{u-v}_{\dB_{2,\infty}^0}
    &
    =
    \n{(-\Delta)^{-1}\mathbb{P} \div (u \otimes (u-v) + (u-v) \otimes v)}_{\dB_{2,\infty}^0}
    \\
    &
    \leq 
    C\n{u \otimes (u-v) + (u-v) \otimes v}_{\dB_{2,\infty}^{-1}}
    \\
    &
    \leq 
    C\sp{\n{u}_{\dB_{2,1}^0} + \n{v}_{\dB_{2,1}^0}}\n{u-v}_{\dB_{2,\infty}^0},
\end{align}
which implies $u \equiv v$, provided that $u$ and $v$ are small in $\dB_{2,1}^0(\R^2)$.
Note that the above argument does not work if the third index $q=1$ is replaced by $q>1$.

The aim of this manuscript is to reveal the non-uniqueness mechanism in $L^2$, where the previous results such as \cite{Rieusset-25-JFA,ABGN25,Che-Hou-26} did not achieved, in the whole plane case.
Moreover, we claim non-uniqueness not only in $L^2$ but also in sharp Besov spaces $\dB_{2,q}^0(\R^2)$ with $q>1$.
Note that the sharpness follows from uniqueness for $q=1$ as mentioned in above paragraph.
The main result of this paper reads as follows.
\begin{thm}\label{thm:main}
    Let $n=2$, $q>1$, and $\varepsilon>0$.
    Then,
    there exists a divergence-free external force $f \in \dB_{2,q}^{-2}(\R^2)$ generating two distinct solutions $u^+,u^- \in \dB_{2,q}^0(\R^2)$ to \eqref{eq:sNS} on $\R^2$ and  
    \begin{align}
        \| u^+ \|_{\dB_{2,q}^0} + \| u^- \|_{\dB_{2,q}^0} + \| f \|_{\dB_{2,q}^{-2}} < \varepsilon.
    \end{align}
\end{thm}
We outline\footnote{To simplify the presentation, some of the notation in this section differs from that used in subsequent sections.} the proof of Theorem \ref{thm:main}. 
The author's original plan was to use the method proposed in the author's previous
work \cite{Fujii-26}, in which one incorporates into each building block of a convex integration scheme a flow for which the relevant nonlinear estimate breaks
down.  In this method, however, the condition $q>2$ is essential in Besov spaces with zero differential index.  
It therefore does not seem suited to proving the failure of uniqueness in $L^2$, which is the case of $q=2$, by this method.  
We therefore do not use convex integration, but instead adapt the Golovkin method \cite{Gol-64} for proving non-uniqueness of solutions to elliptic equations with quadratic nonlinear terms.
See \cite{Mengual-Solera-2026-arXiv,Castro-Faraco-Mengual-Solera-2025-arXiv} for recent non-uniqueness results for the nonstationary Navier--Stokes and SQG equations obtained by the Golovkin method. 
We note that those works are based on instability of self-similar solutions,
whereas our construction is different.

The Golovkin method focuses on the elementary fact that, if the quadratic equation
$x+ax^2=b$ has two distinct roots, then they may be written in the
symmetric form $x=\alpha\pm\beta$ with $\beta \neq 0$.  In the same spirit, we look for two
solutions $u^+$ and $u^-$ of \eqref{eq:sNS} with a single external force $f$ of the form
\begin{align}\label{eq:idea-pair}
    u^+=v+\varepsilon w,
    \qquad
    u^-=v-\varepsilon w
\end{align}
with $w \not \equiv 0$.
Under this notation, \eqref{eq:sNS} is reformulated as
\begin{align}
    &-\Delta w
    +\mathbb P\bigl((v\cdot\nabla)w+(w\cdot\nabla)v\bigr)=0,
    \label{eq:idea-Oseen}
    \\
    &-\Delta v
    +\mathbb P\bigl((v\cdot\nabla)v
    +\varepsilon^2(w\cdot\nabla)w\bigr)=f.
    \label{eq:idea-average}
\end{align}
Thus, the first task is to determine a suitable small $v\in\dB_{2,q}^0(\R^2)$ and construct a nonzero solution
$w\in\dB_{2,q}^0(\R^2)$ of the linear Oseen equation
\eqref{eq:idea-Oseen}.  
We then use the second equation to define a divergence-free force
$f$ and prove that it is small in $\dB_{2,q}^{-2}(\R^2)$.  
Notice that a nontrivial Oseen solution $w$ for some arbitrarily small $v$ could not exist if one had
the bilinear estimate
\begin{align}\label{eq:idea-forbidden-estimate}
    \|\mathbb P(a\cdot\nabla)b\|_{\dB_{2,q}^{-2}}
    =
    \|\mathbb P\div(a\otimes b)\|_{\dB_{2,q}^{-2}}
    \leq 
    C
    \|a\|_{\dB_{2,q}^0}
    \|b\|_{\dB_{2,q}^0}
\end{align}
wit divergence-free vector fields $a,b \in \dB_{2,q}^0(\R^2)$.
Thus, our construction essentially must depend on the failure of
\eqref{eq:idea-forbidden-estimate}.

To analyze the Oseen equation, we rewrite it in scalar form on the Fourier side.  
Since the Fourier transform of a divergence-free vector field is parallel to $\xi^\perp$, we pass from $v$ and $w$ to scalar functions $\phi$ and $\psi$ by setting
\begin{align}\label{eq:idea-scalarization}
    \phi(\xi):=i\frac{\xi^\perp}{|\xi|}\cdot\widehat v(\xi),
    \qquad 
    \psi(\xi):=i\frac{\xi^\perp}{|\xi|}\cdot\widehat w(\xi).
\end{align}
After applying $(-\Delta)^{-1}$ and rewriting \eqref{eq:idea-Oseen} in terms of $\phi$ and $\psi$, we see that
\begin{align}\label{eq:idea-scalar-oseen}
    \psi+\mathcal O[\phi,\psi]=0,
\end{align}
where $\mathcal O$ is the bilinear Fourier multiplier
\begin{align}\label{eq:idea-Oseen-operator}
    \mathcal O[\phi,\psi](\xi)
    =\frac{1}{2\pi|\xi|}\int_{\R^2}
    \Gamma(\xi,\eta)
    \phi(\xi-\eta)\psi(\eta)\,d\eta,
\end{align}
with some homogeneous function $\Gamma(\xi,\eta)$.
Since, in our analysis, the principal contribution to \eqref{eq:idea-Oseen-operator} comes from the high$\times$high$ \to $low interaction $|\xi|\ll|\eta|$, the multiplier is expanded by $\Gamma(\xi,\eta)=\sin (2(\theta_\eta - \theta_\xi)) + O(|\xi|/|\eta|)$, where $\theta_\xi$ denotes the polar angle of $\xi$.
Consequently, the leading part of the Oseen operator $\mathcal{O}$ has the form
\begin{align}\label{eq:idea-angular-decomposition}
    \begin{split}
    \mathcal O[\phi,\psi](\xi)
    ={}&
    \frac{\cos(2\theta_\xi)}{2\pi|\xi|}
    \int_{\R^2}\sin(2\theta_\eta)\phi(-\eta)\psi(\eta)\,d\eta
    \\
    &-
    \frac{\sin(2\theta_\xi)}{2\pi|\xi|}
    \int_{\R^2}\cos(2\theta_\eta)\phi(-\eta)\psi(\eta)\,d\eta
    +\text{lower order  terms}
    \\
    =:{}&
    \mathcal O_1[\phi,\psi](\xi)+\mathcal O_2[\phi,\psi](\xi)+\text{lower order  terms},
    \end{split}
\end{align}
provided that $|\xi| \ll |\eta|$ for all $\eta \in \supp \psi$.
To construct appropriate $\phi$ and $\psi$, we choose a positive smooth function $\chi=\chi(r) \in C_c^\infty((0,\infty))$ satisfying $\chi \equiv 1$ on some neighborhood of $r=1$ and
$\sum_{j\in\Z}\chi(2^{-j}r)=1$ for $r>0$.  Choose a smooth function $\rho =\rho(r)$ supported in a smaller annulus on which $\chi=1$.  For each $j\in\Z$, define
\begin{align}\label{eq:idea-building-blocks}
    V_j(\xi):=\rho(2^{-j}|\xi|)
    \frac{\sin(4\theta_\xi)}{|\xi|},
    \qquad
    W_j(\xi):=\chi(2^{-j}|\xi|)
    \frac{\cos(2\theta_\xi)}{|\xi|}.
\end{align}
Under the notation introduced in Section \ref{sec:pre}, these functions satisfy
$V_j\in\widetilde L^2_{\rm sym}$ and $W_j\in L^2_{\rm sym}$.  
The change of variables $(\eta_1,\eta_2) \mapsto (-\eta_1,\eta_2)$ makes $\mathcal{O}_2[V_j,W_j]$ vanish.  Moreover, normalizing $\rho$ suitably, we see that
\begin{align}\label{eq:idea-block-action}
    \chi(2^{-k}|\xi|)\mathcal O_1[V_j,W_j](\xi)=W_k(\xi),
    \qquad j>k.
\end{align}
Our plan is to sum $V_j$ and $W_j$ with respect to $j \in \Z$ to find the desired $\phi$ and $\psi$.
To this end, the choice of coefficients is crucial.
Let real sequences
$\nu=\{\nu_j\}_{j\in\Z}$ and
$\omega=\{\omega_j\}_{j\in\Z}$ in $\ell^q (\Z) \setminus \ell^1(\Z)$ satisfy $|\nu_j|, |\omega_j| \simeq |j|^{-1}$, $\|\nu\|_{\ell^q}\simeq\varepsilon$, $\|\omega\|_{\ell^q}\simeq 1$, and
\begin{align}\label{eq:idea-sequence-model}
    \omega_k-\sum_{j>k}\nu_j\omega_j=0,
    \qquad k\in\Z.
\end{align}
Note that \eqref{eq:idea-sequence-model} is the sequence model for the scalar equation \eqref{eq:idea-scalar-oseen}, with the convection term simplified to the high$\times$high$ \to $low interaction.
By \eqref{eq:idea-sequence-model}, we now determine $\phi$ and the principal part of $\psi$ as
\begin{align}\label{eq:idea-principal-functions}
    \phi(\xi):=-\sum_{j\in\Z}\nu_jV_j(\xi),
    \qquad
    \psi^{({\rm p})}(\xi):=\sum_{j\in\Z}\omega_jW_j(\xi).
\end{align}
We henceforth fix this $\phi$ and the corresponding vector field $v=\mathscr{F}^{-1}[-i(\xi^\perp/|\xi|)\phi(\xi)]$.  
Absorbing all product interactions other than the high$\times$high$ \to $low interactions into the harmless remainder, we obtain
\begin{align}\label{eq:idea-formal-kernel}
    &\psi^{({\rm p})}(\xi)+\mathcal O[\phi,\psi^{({\rm p})}](\xi)
    \\
    &\quad 
    =\psi^{({\rm p})}(\xi)-\sum_{k\in\Z}\sum_{j>k}\nu_j\omega_j
    \chi(2^{-k}|\xi|)\mathcal O_1[V_j,W_j](\xi)
    +\text{lower order  terms}
    \\
    &\quad 
    =\sum_{k\in\Z}\omega_kW_k(\xi)-\sum_{k\in\Z}\sum_{j>k}\nu_j\omega_jW_k(\xi)
    +\text{lower order  terms}
    \\
    &\quad 
    =\text{lower order  terms}.
\end{align}

In order to turn this formal calculation into an exact solution, we work in the function space
$X^q_{\rm sym}$ consisting of all scalar fields in $L^2_{\rm sym}$ whose corresponding vector field belongs to $\dB_{2,q}^0(\R^2)$ and which have a half-order $L^2$ H\"older regularity for Fourier variables in dyadic shells.   
Then, we show that there exists a small perturbation $\widetilde\psi$ such that $\psi=\psi^{({\rm p})}+\widetilde\psi$ solves \eqref{eq:idea-Oseen} with the estimate $\|\widetilde\psi\|_{X^q}\ll \|\psi^{({\rm p})}\|_{X^q}$.
As we may check that $\|\psi^{({\rm p})}\|_{X^q} \simeq 1$ and $\| \phi \|_{X^q} \lesssim \varepsilon$, 
both $\psi$ and $w=\mathscr{F}^{-1}[-i(\xi^\perp/|\xi|)\psi(\xi)]$ are nonzero, and we have $\| u^\pm \|_{\dB_{2,q}^0} \leq \| v \|_{\dB_{2,q}^0} + \varepsilon\| w \|_{\dB_{2,q}^0} \lesssim \varepsilon$.
This completes the analysis for the Oseen equations \eqref{eq:idea-Oseen}.

It remains to verify that the force determined by \eqref{eq:idea-average}
is small in $\dB_{2,q}^{-2}(\R^2)$.  
Although $v$ and $\varepsilon w$ are small in $\dot B_{2,q}^0(\R^2)$, this is not a trivial fact since the paraproduct is not bounded from $\dB_{2,q}^0(\R^2)\times\dB_{2,q}^0(\R^2)$ to $\dB_{2,q}^{-1}(\R^2)$.
The key ingredients to overcome this are the symmetry of the flow and the $L^2$ H\"older regularity ensured by $X^q$ norm.  
Since the symmetric condition for $w$ following from $\psi \in X^q_{\rm sym}$ implies that
$\mathscr{F}[\Delta_k w \otimes \Delta_\ell w](0)$ is a scalar matrix, 
we see that $i\widehat{\mathbb P}(\xi)\mathscr{F}[\Delta_k w \otimes \Delta_\ell w](0)\xi=0$, which yields
\begin{align}
    &
    \mathscr{F}\lp{\mathbb P \div(\Delta_k w \otimes \Delta_\ell w)}(\xi)
    =
    i\widehat{\mathbb P}(\xi)\mathscr{F}[\Delta_k w \otimes \Delta_\ell w](\xi)\xi
    \\
    &\quad 
    =
    i\widehat{\mathbb P}(\xi)
    \bigl(\mathscr{F}[\Delta_k w \otimes \Delta_\ell w](\xi)-\mathscr{F}[\Delta_k w \otimes \Delta_\ell w](0)\bigr)\xi
    \\
    &\quad =
    i\widehat{\mathbb P}(\xi)
    \left(
    \int_{\R^2}
    (\widehat{\Delta_k w}(\xi-\eta)-\widehat{\Delta_k w}(-\eta)) \otimes \widehat{\Delta_\ell w}(\eta)\, d\eta
    \right)\xi.
\end{align}
The $L^2$ H\"older regularity provided by $X^q$ norm controls this
difference and yields a summable gain in the high$\times$high$ \to $low
interaction.  
The symmetry of $v$ also gives the same structure
for $v\otimes v$, and hence we have
\begin{align}
    \|\mathbb P\div(v\otimes v)\|_{\dB_{2,q}^{-2}}
    \leq C\|\phi\|_{X^q}^2,
    \qquad
    \|\mathbb P\div(w\otimes w)\|_{\dB_{2,q}^{-2}}
    \leq C\|\psi\|_{X^q}^2.
\end{align}
Since $\|\phi\|_{X^q}\lesssim\varepsilon$ and
$\|\psi\|_{X^q}\simeq1$, we obtain
$\|f\|_{\dB_{2,q}^{-2}}\lesssim \varepsilon$ by \eqref{eq:idea-average}, and complete the proof.

This paper is organized as follows.
In Section \ref{sec:pre}, we prepare useful tools for our analysis.
In Section \ref{sec:Oseen}, we establish the linear analysis.
Then, we prove the main result in Section \ref{sec:pf}.
\section{Preliminaries}\label{sec:pre}
In this section, we introduce notation and elementary facts, and establish the simplified model mentioned in the preceding section.
\subsection{Notations}
We first summarize some notations used in this paper.
For given Schwartz functions $\varphi$, $\psi$ on $\R^2$, the Fourier transform and its inverse are defined by 
\begin{align}
    \mathscr{F}[\varphi](\xi)=\widehat{\varphi}(\xi)&:=\frac{1}{2\pi}\int_{\R^2} e^{-ix \cdot \xi}\varphi(x)\, dx,
    \\
    \mathscr{F}^{-1}[\psi](x)&:=\frac{1}{2\pi}\int_{\R^2} e^{ix \cdot \xi}\psi(\xi)\, d\xi.
\end{align}
Next, we prepare some function spaces.
Let us first define
\begin{align}
    L^2_{\rm sym}
    &:=\Mp{\varphi \in L^2(\R^2;\R)\ ;\ 
    \begin{aligned}
        &\varphi(-\xi_2,\xi_1)=-\varphi(\xi_1,\xi_2),\\ 
        &\varphi(-\xi_1,\xi_2)=\varphi(\xi_1,\xi_2)    
    \end{aligned}
    },
    \\
    \widetilde{L}^2_{\rm sym}
    &:=\Mp{\varphi \in L^2(\R^2;\R)\ ;\ 
    \begin{aligned}
        &\varphi(-\xi_2,\xi_1)=\varphi(\xi_1,\xi_2),\\ 
        &\varphi(-\xi_1,\xi_2)=-\varphi(\xi_1,\xi_2)    
    \end{aligned}
    }.
\end{align}
Let $\chi \in C_c^\infty([0,\infty);[0,1])$ satisfy
\begin{align}
    \supp \chi \subset \lp{2^{-\frac{3}{4}},2^{\frac{3}{4}}},
    \qquad
    \chi(r) = 1 \quad {\rm for\ all\ }
    r \in \lp{2^{-\frac{1}{4}},2^{\frac{1}{4}}},
\end{align}
and 
\begin{align}
    \sum_{j \in \Z} \chi(2^{-j}r)=1 \qquad {\rm for\ all\ }r>0.
\end{align}
We call $\{\chi_j(\xi):=\chi(2^{-j}|\xi|)\}_{j \in \Z}$ the Littlewood--Paley decomposition and define $\Delta_jf:=\mathscr{F}^{-1}[\chi_j(\xi)\widehat{f}(\xi)]$ for tempered distributions $f$.
Recall that the Besov norm is defined by 
\begin{align}
    \| g \|_{\dB_{p,q}^s}
    :=
    \n{\Mp{2^{sj}\| \Delta_j g \|_{L^p}}_{j \in \Z}}_{\ell^q}
\end{align}
for $1 \leq p,q \leq \infty$ and $s \in \R$.
\begin{df}
For $1 \leq q \leq 2$, we define
\begin{align}
    X^q
    &:=\Mp{\psi \in L^2(\R^2)\ ;\ \n{\psi}_{X^q}<\infty},
    \\
    \n{\psi}_{X^q}
    &:=
    \Mp{\sum_{j \in \Z}\sp{\n{\chi_j\psi}_{L^2}+\sup_{\eta \in \R^2 \setminus \{0\}}
    \frac{\n{(\chi_j\psi)(\cdot+\eta)-(\chi_j\psi)(\cdot)}_{L^2}}{\sp{\min \{1,2^{-j}|\eta|\}}^{1/2}}}^q}^{\frac{1}{q}}.
\end{align}
Moreover, we denote
\begin{align}
    X^q_{\rm sym}
    :=\Mp{\psi \in L^2_{\rm sym}\ ;\ \n{\psi}_{X^q}<\infty},
    \quad 
    \widetilde{X}^q_{\rm sym}
    :=\Mp{\psi \in \widetilde{L}^2_{\rm sym}\ ;\ \n{\psi}_{X^q}<\infty}.
\end{align}
Note that $\n{\varphi}_{\dB_{2,q}^0} \leq \|\widehat{\varphi}\|_{X^q}$ holds for all $\varphi \in L^2(\R^2)$ with $\widehat{\varphi} \in X^q$.   
\end{df}
For $\xi \in \R^2\setminus \{0\}$, 
we define $\theta_\xi \in \T:=\R/(2\pi \Z)$ via 
\begin{align}
    \cos \theta_\xi = \frac{\xi_1}{|\xi|},
    \qquad
    \sin \theta_\xi = \frac{\xi_2}{|\xi|}.
\end{align}
For $\psi\in L^2(\R^2)$, we define 
\begin{align}
    \mathcal{U}[\psi]:=\mathscr{F}^{-1}\lp{-i \psi(\xi) \frac{\xi^\perp}{|\xi|}},
    \qquad
    \xi^\perp:=(-\xi_2,\xi_1).
\end{align}
Note that $\mathcal{U}[\psi]$ is a $\R^2$-valued divergence-free vector field in $L^2(\R^2)$ if $\psi\in L^2_{\rm sym} \cup \widetilde L^2_{\rm sym}$.
Conversely, every divergence-free vector field $u \in L^2(\R^2)$ is written as $u=\mathcal{U}[\psi]$, where
\begin{align}
    \psi(\xi):=i \frac{\xi^\perp}{|\xi|} \cdot \widehat{u}(\xi).
\end{align}

Finally, we define building blocks of our non-unique solutions.
Let $\rho \in C_c^{\infty}([0,\infty);[0,\infty))$ satisfy
\begin{align}
    \int_0^\infty \frac{\rho(r)}{r}\, dr = 4,
    \qquad
    \supp \rho \subset \lp{2^{-\frac{1}{8}},2^{\frac{1}{8}}}.
\end{align}
For $j \in \Z$, we define $\rho_j(\xi):=\rho(2^{-j}|\xi|)$.
By the support relation, it holds $\rho_j(\xi)=\chi_j(\xi)\rho_j(\xi)$ for all $\xi \in \R^2$.
For $j \in \Z$, we define 
\begin{align}
    V_j(\xi):=\rho_j(\xi)\frac{\sin (4\theta_\xi)}{|\xi|},
    \qquad
    W_j(\xi):=\chi_j(\xi)\frac{\cos (2\theta_\xi)}{|\xi|}.
\end{align}
A direct calculation yields $V_j \in \widetilde L_{\rm sym}^2$ and $W_j \in L^2_{\rm sym}$ for all $j \in \Z$.
By $V_j(\xi)=2^{-j}V_0(2^{-j}\xi)$ and $W_j(\xi)=2^{-j}W_0(2^{-j}\xi)$, we see that
\begin{align}
    \n{\nabla^m V_j}_{L^p}
    &=
    2^{(\frac{2}{p}-1-m)j}\n{\nabla^m V_0}_{L^p},
    \\
    \n{\nabla^m W_j}_{L^p}
    &=
    2^{(\frac{2}{p}-1-m)j}\n{\nabla^m W_0}_{L^p}
\end{align}
for all $1 \leq p \leq \infty$ and $m \in \N \cup \{0\}$.
We also use the following notation:
\begin{align}
    \mathcal{W}p(\xi)
    :=
    \sum_{j \in \Z}p_jW_j(\xi),
    \qquad p=\{p_j\}_{j \in \Z} \in \R^\Z.
\end{align}
The above right hand side is finite for fixed $\xi$, since the only finitely many $W_j$ are nonzero.
\subsection{Elementary lemmas}
\begin{lemm}
    There exists a positive constant $C$ such that 
    \begin{align}
        \n{\widehat{\mathcal{U}[\psi]}}_{X^q}
        \leq 
        C
        \n{\psi}_{X^q}
    \end{align}
    for all $1 \leq q \leq 2$ and $\psi \in X^q$.
\end{lemm}
\begin{proof}
    Since we have
    \begin{align}
        \sp{
        \sum_{j \in \Z}
        \n{\chi_j \widehat{\mathcal{U}[\psi]}}_{L^2}^q
        }^{\frac{1}{q}}
        =
        \sp{
        \sum_{j \in \Z}
        \n{\chi_j \psi}_{L^2}^q
        }^{\frac{1}{q}},
    \end{align}
    it suffices to focus on the estimate of the difference.
    As we see that 
    \begin{align}
        &
        \sp{\chi_j\widehat{\mathcal{U}[\psi]}}(\xi+\eta)
        -
        \sp{\chi_j\widehat{\mathcal{U}[\psi]}}(\xi)
        \\
        &\quad
        =
        -i\frac{(\xi+\eta)^\perp}{|\xi+\eta|}\chi_j(\xi+\eta)\psi(\xi+\eta)
        +i\frac{\xi^\perp}{|\xi|}\chi_j(\xi)\psi(\xi)\\
        &\quad 
        =
        -i\frac{(\xi+\eta)^\perp}{|\xi+\eta|}\sp{\chi_j(\xi+\eta)\psi(\xi+\eta)-\chi_j(\xi)\psi(\xi)}
        -i\sp{\frac{(\xi+\eta)^\perp}{|\xi+\eta|}-\frac{\xi^\perp}{|\xi|}}\chi_j(\xi)\psi(\xi)
    \end{align}
    and 
    \begin{align}
        \abso{\frac{\xi+\eta}{|\xi+\eta|}-\frac{\xi}{|\xi|}} 
        &
        =
        \abso{\int_0^1 \left.(\eta \cdot \nabla_\zeta)\frac{\zeta}{|\zeta|} \right|_{\zeta=\xi+\theta\eta}\, d\theta}
        \\
        &
        \leq 
        \int_0^1\frac{|\eta|}{|\xi+\theta\eta|}\, d\theta
        \leq \frac{|\eta|}{|\xi|-|\eta|}
        \leq C2^{-j}|\eta|
    \end{align}
    for $|\eta| \leq 2^{j-1}$, it follows that 
    \begin{align}
        \abso{\sp{\chi_j\widehat{\mathcal{U}[\psi]}}(\xi+\eta)
        -
        \sp{\chi_j\widehat{\mathcal{U}[\psi]}}(\xi)}
        \leq{}& 
        |(\chi_j\psi)(\xi+\eta) - (\chi_j\psi)(\xi)|
        \\
        &
        +
        C\sp{\min\{1,2^{-j}|\eta|\}}
        |(\chi_j\psi)(\xi)|.
    \end{align}
    Taking $L^2$-norm with respect to $\xi \in \R^2$, we have 
    \begin{align}
        \frac{\n{(\chi_j\widehat{\mathcal{U}[\psi]})(\cdot+\eta)-(\chi_j\widehat{\mathcal{U}[\psi]})(\cdot)}_{L^2}}
        {\sp{\min\{1,2^{-j}|\eta|\}}^{1/2}}
        \leq 
        \frac{\n{(\chi_j\psi)(\cdot+\eta) - (\chi_j\psi)(\cdot)}_{L^2}}{\sp{\min\{1,2^{-j}|\eta|\}}^{1/2}}
        +
        C\n{\chi_j\psi}_{L^2}.
    \end{align}
    We take $\ell^q$-norm to complete the proof.
\end{proof}
\begin{lemm}
    For $1 \leq q \leq 2$, $\mathcal{W}$ is a bounded linear operator from $\ell^q(\Z)$ to $X^q_{\rm sym}$.
\end{lemm}
\begin{proof}
    For $p=\{p_j\}_{j \in \Z} \in \ell^q(\Z)$, we see by the almost orthogonality of $\{\chi_j\}_{j \in \Z}$ that 
    \begin{align}
        \chi_j(\xi)(\mathcal{W}p)(\xi) = \sum_{|j'-j| \leq 1} \chi_j(\xi)W_{j'}(\xi)p_{j'}.
    \end{align}
    Thus, we have 
    \begin{align}
        \sp{\sum_{j \in \Z}\n{\chi_j\mathcal{W}p}_{L^2}^q}^{\frac{1}{q}} 
        \leq 
        C
        \sp{\sum_{j \in \Z}\n{W_j}_{L^2}^q|p_j|^q}^{\frac{1}{q}}
        =
        C\n{W_0}_{L^2}\n{p}_{\ell^q}.
    \end{align}
    For the estimate of the difference, we see that 
    \begin{align}
        &
        \n{(\chi_j\mathcal{W}p)(\cdot+\eta) - (\chi_j\mathcal{W}p)(\cdot)}_{L^2}
        \leq 
        \sum_{|j'-j| \leq 1} \n{(\chi_{j}W_{j'})(\cdot+\eta) - (\chi_{j}W_{j'})(\cdot)}_{L^2}|p_{j'}|
        \\
        &\quad 
        \leq 
        \sum_{|j'-j| \leq 1} \min
        \Mp{
        2\n{\chi_{j}W_{j'}}_{L^2},
        |\eta|
        \n{\nabla(\chi_{j}W_{j'})}_{L^2}
        }|p_{j'}|
        \\
        &\quad
        \leq 
        C\sp{\min \Mp{1, 2^{-j}|\eta|}}^{\frac{1}{2}}
        \sum_{|j'-j| \leq 1} 
        |p_{j'}|,
    \end{align}
    which implies 
    \begin{align}
        \Mp{
        \sum_{j \in \Z}
        \sp{
        \sup_{\eta \neq 0}
        \frac{\n{\chi_j\mathcal{W}p(\cdot+\eta) - \chi_j\mathcal{W}p(\cdot)}_{L^2}}{\sp{\min \Mp{1, 2^{-j}|\eta|}}^{\frac{1}{2}}}
        }^q
        }^{\frac{1}{q}}
        \leq C\n{p}_{\ell^q}.
    \end{align}
    This completes the proof.
\end{proof}
\subsection{A simplified model on sequence spaces}
Let $N \in \N$ and define two sequences $\omega_N=\{\omega_{N,j}\}_{j \in \Z}$ and $\nu_N=\{\nu_{N,j}\}_{j \in \Z}$ by 
\begin{align}
    \omega_{N,j}:= \frac{N}{N+|j|},
    \qquad
    \nu_{N,j}:= \frac{|j|-|j-1|}{N+|j-1|}.
\end{align}
We then easily see by $\nu_{N,j}\omega_{N,j}=\omega_{N,j-1}-\omega_{N,j}$ that 
\begin{align}
    \sum_{j \in \Z} \nu_{N,j}\omega_{N,j} = 0,
    \qquad
    \sum_{j > k}\nu_{N,j}\omega_{N,j} = \omega_{N,k} \quad (k \in \Z),
\end{align}
and both series above converges absolutely.
For the estimate of $\omega_N$, we have 
\begin{align}\label{}
    \n{\omega_N}_{\ell^q} 
    \simeq
    \Mp{\int_{\R} \sp{\frac{N}{N+|t|}}^q\, dt}^{\frac{1}{q}} 
    = 
    \begin{cases}
        \sp{\frac{2}{q-1}}^{\frac{1}{q}} N^{\frac{1}{q}} & (q>1),
        \\
        \infty & (q=1).
    \end{cases}
\end{align}
For $1 < q <\infty$, we define 
\begin{align}
    \ell_{\nu_N}^q(\Z)
    :=
    \Mp{p=\{p_j\}_{j \in \Z} \in \ell^q(\Z)\ ;\ \sum_{j \in \Z} \nu_{N,j}p_j = 0 }
\end{align}
and consider a discrete operator $H_N:\ell_{\nu_N}^q(\Z) \to \R^{\Z}$ defined by
\begin{align}
    (H_Np)_k:=\sum_{j>k}\nu_{N,j}p_j, 
    \qquad 
    k \in \Z,\quad p=\{p_j\}_{j \in \Z} \in \ell^q_{\nu_N}(\Z).
\end{align}
Note that the series defining $(H_Np)_k$ converges absolutely since $\nu_N \in \ell^{q'}(\Z)$.
We may show that $H_N$ is a bounded linear operator from $\ell_{\nu_N}^q(\Z)$ to $\ell^q(\Z)$.
\begin{lemm}\label{lemm:H_N}
    For $1 < q < \infty$, there exists a positive constant $C=C(q)$ such that 
    $\| H_Np \|_{\ell^q} \leq C \| p \|_{\ell^q}$
    for all $p \in \ell^q_{\nu_N}(\Z)$ and $N \in \N$.
\end{lemm}
To prove this, we recall the Hardy--Copson inequality.
\begin{lemm}\label{lemm:Hardy}
    For $1 < q <\infty$, there exists a positive constant $C=C(q)$ such that
    \begin{align}
        &
        \n{\Mp{\frac{1}{k}\sum_{\ell=1}^k p_\ell }_{k \in \N} }_{\ell^q}
        \leq 
        C
        \n{p}_{\ell^q},
        \qquad 
        \n{ \Mp{\sum_{\ell=k}^\infty \frac{p_\ell}{\ell} }_{k \in \N} }_{\ell^q}
        \leq 
        C
        \n{p}_{\ell^q}
    \end{align}
    for all $p=\{p_\ell\}_{\ell \in \N} \in \ell^q(\N)$.
\end{lemm}
See \cite[Theorem 326, p.~239]{Hardy-Littlewood-Polya-1952} for the proof of Lemma \ref{lemm:Hardy}.
\begin{proof}[Proof of Lemma \ref{lemm:H_N}]
We first consider the case of $k\geq 0$.
Since
\begin{align}
    \abso{(H_Np)_k}
    \leq
    \sum_{j>k}
    \frac{|p_j|}{N+j-1}
    \leq 
    \sum_{\ell=k+1}^\infty 
    \frac{|p_\ell|}{\ell},
\end{align}
Lemma \ref{lemm:Hardy} yields
\begin{align}
    \sp{\sum_{k=0}^\infty |(H_Np)_k|^q}^{\frac{1}{q}}
    &\leq 
    \Mp{\sum_{k=0}^\infty
    \sp{\sum_{\ell=k+1}^\infty 
    \frac{|p_{\ell}|}{\ell}}^q
    }^{\frac{1}{q}}
    \leq 
    C\sp{\sum_{j=0}^\infty|p_j|^q}^{\frac{1}{q}}.
\end{align}
For the case of $k \leq -1$, we see by $\sum_{j \in \Z}\nu_{N,j}p_j=0$ that 
\begin{align}
    |(H_Np)_k| = \abso{\sum_{j \leq k}\frac{p_j}{N-j+1}}
    \leq 
    \sum_{\ell=-k}^{\infty} \frac{|p_{-\ell}|}{\ell}.
\end{align}
Taking $\ell^q$-norm with respect to $k$ and using Lemma \ref{lemm:Hardy}, we have 
\begin{align}
    \sp{\sum_{k \leq -1} |(H_Np)_k|^q}^{\frac{1}{q}}
    \leq \Mp{\sum_{-k=1}^{\infty}\sp{\sum_{\ell=-k}^{\infty} \frac{|p_{-\ell}|}{\ell}}^q}^{\frac{1}{q}}
    \leq C \sp{\sum_{\ell=1}^{\infty}|p_{-\ell}|^q}^{\frac{1}{q}},
\end{align}
which completes the proof.
\end{proof}
Let $L_N:=\Id - H_N$.
By Lemma \ref{lemm:H_N}, $L_N$ is bounded from $\ell_{\nu_N}^q(\Z)$ to $\ell^q(\Z)$.
\begin{lemm}\label{lemm:Ker_L_N}
    It holds $\operatorname{ker}L_N=\operatorname{span}\{\omega_N\}$.
\end{lemm}
\begin{proof}
    Since $H_N\omega_N=\omega_N$, it suffices to show $\operatorname{ker}L_N \subset \operatorname{span}\{\omega_N\}$.
    Let $p \in \ell_{\nu_N}^q(\mathbb Z)$ satisfy $L_Np=0$, that is $H_Np=p$.
    Then, since $\nu_{N,j}p_j=(H_Np)_{j-1}-(H_Np)_j=p_{j-1} - p_j$, we have $(1+\nu_{N,j})p_j=p_{j-1}$.
    As $1+\nu_{N,j}=(N+|j|)/(N+|j-1|)=\omega_{N,j-1}/\omega_{N,j}$, it holds $p_j/\omega_{N,j} = p_{j-1}/\omega_{N,j-1}$, which implies $p \in \operatorname{span}\{\omega_N\}$.
    This completes the proof.
\end{proof}
Next, we aim to find an appropriate right inverse of $L_N$.
For given $y=\{y_j\}_{j\in \Z} \in \ell^q(\Z)$, we assume that $p=\{p_j\}_{j \in \Z} \in \ell_{\nu_N}^q(\Z)$ satisfies $L_Np=y$. 
Then, since 
\begin{align}
    p_j-\sum_{\ell>j}\nu_{N,\ell}p_\ell = y_j, \qquad j \in \Z,
\end{align}
we have $p_j - p_{j-1} + \nu_{N,j}p_j = y_j - y_{j-1}$, that is $(1+\nu_{N,j})p_j - y_j = p_{j-1} - y_{j-1}$.
As $1+\nu_{N,j}=(N+|j|)/(N+|j-1|)$, we see that 
\begin{align}
    (N+|j|)(p_j-y_j)=(N+|j-1|)(p_{j-1}-y_{j-1})-(|j|-|j-1|)y_j,
\end{align}
which yields
\begin{align}
    (N+|j|)(p_j-y_j)
    = 
    \begin{cases}
        N(p_0-y_0)-\displaystyle\sum_{\ell=1}^jy_\ell, & (j \geq 1), \\
        N(p_0-y_0)-\displaystyle\sum_{\ell=j+1}^0y_\ell, & (j \leq -1).
    \end{cases}
\end{align}
Hence, combining the above calculation with Lemma \ref{lemm:Hardy}, we obtain the following lemma.
\begin{lemm}\label{lemm:R_N}
    Let $1<q<\infty$.
    Let $R_N:\ell^q(\Z) \to \R^{\Z}$ be defined as
    \begin{align}
        (R_Ny)_j
        :=
        \begin{cases}
        y_j-\dfrac{1}{N+j}\displaystyle\sum_{\ell=1}^jy_\ell, & (j \geq 1), \\
        y_0, & (j=0), \\
        y_j-\dfrac{1}{N-j}\displaystyle\sum_{\ell=j+1}^0y_\ell, & (j \leq -1).
    \end{cases}
    \end{align}
    Then, there exists a positive constant $C=C(q)$ such that 
    \begin{align}
        \n{R_Ny}_{\ell^q} \leq C \n{y}_{\ell^q}, 
        \qquad
        R_Ny \in \ell_{\nu_N}^q(\Z), 
        \qquad
        L_NR_Ny=y
    \end{align}
    for all $y \in \ell^q(\Z)$.
\end{lemm}
\section{Oseen solutions}\label{sec:Oseen}
In this section, we construct nontrivial solutions to the Oseen equations with some background flow.
In what follows, we fix the background flow by 
\begin{align}
    v_N:=\mathcal{U}[\phi_N],
    \qquad
    \phi_N(\xi):=-\sum_{j \in \Z}\nu_{N,j}V_j(\xi),
\end{align}
where the sequence $\nu_N:=\{\nu_{N,j}\}_{j \in \Z}$ is defined in the previous section.
\begin{lemm}\label{lemm:v_N}
    For $1 < q \leq 2$,
    there exists a positive constant $C=C(q)$ such that
    \begin{align}
        \n{v_N}_{\dB_{2,q}^0} \leq CN^{\frac{1}{q}-1},
        \qquad
        \phi_N \in \widetilde{X}^q_{\rm sym}
    \end{align}
    for all $N \in \N$.
\end{lemm}
\begin{proof}
    Since $V_j \in \widetilde{L}^2_{\rm sym}$, we see that $\phi_N \in \widetilde{L}^2_{\rm sym}$.
    For the Besov estimate of $v_N$, 
    it follows from $\chi_j\phi_N=-\nu_{N,j}V_j$ that 
    \begin{align}
        \n{v_N}_{\dB_{2,q}^0}
        &\leq 
        C \n{\phi_N}_{X^q}
        \leq
        C
        \n{V_0}_{H^1}\n{\nu_N}_{\ell^q}
        \\
        &\leq 
        C \sp{\int_{\R}\frac{1}{(N+|t|)^q}\, dt}^{\frac{1}{q}}
        = 
        C\sp{\frac{2}{q-1}}^{\frac{1}{q}}N^{\frac{1}{q}-1},
    \end{align}
    which completes the proof.
\end{proof}
Now, let us consider the Oseen equations
\begin{align}\label{eq:Oseen}
    -\Delta w + \mathbb P \sp{(v_N \cdot \nabla) w + (w \cdot \nabla) v_N} = 0.
\end{align}
The central goal of this section is to show the following proposition.
\begin{prop}\label{prop:Oseen}
For $1 < q \leq 2$, there exist positive constants $N_0=N_0(q) \in \N$ and $C=C(q)$ such that for every $N \in \N$ with $N \geq N_0$, 
\eqref{eq:Oseen} possesses a nontrivial solution $w_N = N^{-1}\mathcal{U}[\psi_N] \in \dB_{2,q}^0(\R^2)$ with some $\psi_N \in X^q_{\rm sym}$
such that 
\begin{align}
    \| w_N \|_{\dB_{2,q}^0} \leq CN^{\frac{1}{q}-1},
    \qquad
    C^{-1}N^{\frac{1}{q}} \leq \| \psi_N \|_{X^q} \leq CN^{\frac{1}{q}}.
\end{align}
\end{prop}
\subsection{Reformulation of the problem}
We focus on the Oseen operator
\begin{align}
    \mathscr{O}_N[w]
    :=
    (-\Delta)^{-1}\mathbb P \sp{(v_N \cdot \nabla) w + (w \cdot \nabla) v_N}.
\end{align}
For any $w=\mathcal{U}[\psi]$ with $\psi \in L_{\rm sym}^2$, 
a scalar field corresponding to $\mathscr{O}_N[w]$ is given by
\begin{align}
    &
    \mathcal{O}_N[\psi](\xi)
    :={}
    i\frac{\xi^\perp}{|\xi|}
    \cdot
    \mathscr{F}\lp{\mathscr{O}_N[w]}(\xi)
    \\
    &
    ={}
    \frac{1}{2\pi|\xi|^2}
    \frac{\xi^\perp}{|\xi|} 
    \cdot 
    \widehat{\mathbb{P}}(\xi)
    \int_{\R^2}
    \frac{\sp{(\xi-\eta)^\perp\cdot \eta}\eta^\perp+ \sp{\eta^\perp \cdot (\xi-\eta)}(\xi-\eta)^\perp}{|\xi-\eta||\eta|}
    \phi_N(\xi-\eta)
    \psi(\eta)
    \, d\eta
    \\
    &
    ={}
    \frac{1}{2\pi|\xi|}
    \int_{\R^2}
    \frac{\sp{(\xi-\eta)^\perp\cdot \eta}\eta \cdot \xi + \sp{\eta^\perp \cdot (\xi-\eta)}(\xi-\eta) \cdot \xi}{|\xi|^2|\xi-\eta||\eta|}
    \phi_N(\xi-\eta)
    \psi(\eta)
    \, d\eta
    \\
    &
    ={}
    -
    \sum_{j \in \Z}
    \nu_{N,j}
    \mathcal{H}_j[\psi](\xi),
\end{align}
where we have set
\begin{align}
    \mathcal{H}_j[\psi](\xi)
    &:=
    \frac{1}{2\pi|\xi|}
    \int_{\R^2}
    \Gamma_\xi(\xi-\eta,\eta)V_j(\xi-\eta)\psi(\eta)\, 
    d\eta,
    \\
    \Gamma_\xi(\zeta,\eta)
    &:=
    \frac{(\eta^\perp \cdot \zeta)(|\zeta|^2-|\eta|^2)}{|\xi|^2|\eta||\zeta|}.
\end{align}
Thus, the Oseen equations
\begin{align}
    -\Delta w + \mathbb P \sp{(v_N \cdot \nabla) w + (w \cdot \nabla) v_N} = 0,
    \qquad
    w=\mathcal{U}[\psi]
\end{align}
are rewritten as 
\begin{align}\label{eq:sc-O}
    \mathcal{L}_N \psi = 0,
    \qquad
    \mathcal{L}_N:=\Id + \mathcal{O}_N.
\end{align}
We may easily check the following lemma.
\begin{lemm}
    For $j,k \in \Z$ and $\psi \in L^2_{\rm sym}$, it holds $\chi_k\mathcal{H}_j[\psi] \in L^2_{\rm sym}$.
\end{lemm}
\subsection{Principal part of the Oseen equations}
The aim of this section is to specify the principal part of the operator $\mathcal{L}_N$.
To begin with, we investigate the properties of the kernel $\Gamma_\xi(\zeta,\eta)$ for $\mathcal{H}_j$.
\begin{lemm}\label{lemm:Gamma}
    Let $\xi,\eta,\zeta \in \R^2 \setminus \{0\}$ with $\xi- \eta \neq 0$.
    Then, the following statements hold true.
    \begin{enumerate}
        \item 
        It holds
        \begin{align}
            \Gamma_\xi(\eta,\zeta)=\Gamma_\xi(\zeta,\eta),
            \qquad 
            |\Gamma_{\xi}(\xi-\eta,\eta)| \leq 1.
        \end{align}
        \item 
        For any orthogonal matrix $T \in \R^{2\times 2}$, it holds
        \begin{align}
            \Gamma_{T\xi}(T\zeta,T\eta) = \det(T) \Gamma_\xi(\zeta,\eta).
        \end{align}
        \item 
        For $0<\delta<1$, there exists a positive constant $C_\delta$, independent of $\xi$, $\eta$, and $\zeta$, such that 
        \begin{align}
            &
            \abso{\Gamma_\xi(\xi-\eta,\eta) - \gamma(\xi,\eta)}
            \leq C_\delta \frac{|\xi|}{|\eta|},
            \\
            &
            \abso{\nabla_\xi \sp{\Gamma_\xi(\xi-\eta,\eta) - \gamma(\xi,\eta)}}
            \leq \frac{C_\delta}{|\eta|},
        \end{align}
        provided 
        $|\xi|/|\eta|\leq \delta$,
        where we have set 
        \begin{align}
            \gamma(\xi,\eta)
            :=
            \sin \sp{2 ( \theta_\eta - \theta_\xi )}.
        \end{align}
        \item 
        There exists a positive constant $C$ such that 
        \begin{align}
            |\nabla_\xi (\Gamma_\xi(\xi-\eta,\eta))|
            &
            \leq 
            C
            \frac{|\xi-\eta|}{|\xi|^2}
            +
            C
            \sp{\frac{1}{|\xi|}+\frac{1}{|\xi-\eta|}}
            \frac{\big||\xi-\eta|^2-|\eta|^2\big|}{|\xi|^2},
            \\
            |\nabla_\xi \gamma(\xi,\eta)|
            &
            \leq 
            \frac{C}{|\xi|}.
        \end{align}
    \end{enumerate}
\end{lemm}
\begin{proof}
    We may see $\Gamma_\xi(\eta,\zeta)=\Gamma_\xi(\zeta,\eta)$ by the definition of $\Gamma_\xi(\zeta,\eta)$. The identity $\Gamma_{T\xi}(T\zeta,T\eta) = \det(T) \Gamma_\xi(\zeta,\eta)$ is proved by $(T\eta)^\perp \cdot (T\zeta)=\det (T\eta, T\zeta)=\det(T)\det(\eta,\zeta)=\det(T)(\eta^\perp\cdot \zeta)$.
    Thus, it remains to prove the other assertions.
    Let 
    \begin{align}
        \lambda=\lambda(\xi,\eta):=\frac{|\xi|}{|\eta|},
        \qquad
        \begin{cases}
        c=c(\xi,\eta):=\dfrac{\xi}{|\xi|} \cdot \dfrac{\eta}{|\eta|}=\cos(\theta_\xi-\theta_\eta),
        \\[10pt]
        s=s(\xi,\eta):=\dfrac{\eta^\perp}{|\eta|} \cdot \dfrac{\xi}{|\xi|} = \sin (\theta_\xi - \theta_\eta ).
        \end{cases}
    \end{align}
    Then, we may rewrite $\Gamma_\xi(\xi-\eta,\eta)$ as
    \begin{align}
        \Gamma_\xi(\xi-\eta,\eta)
        =
        \frac{s(\lambda-2c)}{\sqrt{\lambda^2-2c\lambda+1}}.
    \end{align}
    From $(\sqrt{\lambda^2-2c\lambda+1})^2 - (s(\lambda-2c))^2 = (c\lambda+1-2c^2)^2\geq 0$, we obtain $|\Gamma_\xi(\xi-\eta,\eta)| \leq 1$.

    Next, we consider the estimate of 
    \begin{align}
        \Gamma_\xi(\xi-\eta,\eta)-\gamma(\xi,\eta)
        =
        \frac{s(\lambda-2c)}{\sqrt{\lambda^2-2c\lambda+1}} + 2sc=:\widetilde{\Gamma}(\lambda,c,s).
    \end{align}
    Then, since 
    \begin{align}
        \widetilde{\Gamma}(\lambda,c,s)
        =
        \int_0^\lambda
        \frac{s(cr+1-2c^2)}{(r^2-2cr+1)^{3/2}}
        \, dr,
    \end{align}
    we have for $0<\lambda\leq \delta$ that
    \begin{align}
        \abso{\widetilde{\Gamma}(\lambda,c,s)}
        \leq \int_0^\lambda \frac{4}{(r^2-2r+1)^{3/2}}\, dr = \int_0^\lambda\frac{4}{(1-r)^3}\, dr \leq \frac{4\lambda}{(1-\delta)^3},
    \end{align}
    which provides the first estimate of (3).
    Combining the derivatives
    \begin{align}
        &
        \partial_\lambda \widetilde{\Gamma}(\lambda,c,s) = \frac{s(c\lambda+1-2c^2)}{(\lambda^2-2c\lambda+1)^{3/2}},
        \qquad
        \nabla_\xi \lambda = \frac{\xi}{|\xi||\eta|},
        \\
        &
        \partial_c 
        \widetilde{\Gamma}(\lambda,c,s)
        =
        \int_0^\lambda
        \frac{s\sp{r^3-3cr^2+(4+2c^2)r-4c}}{(r^2-2cr+1)^{5/2}}
        \, dr,
        \qquad
        \nabla_\xi c = \frac{\eta}{|\xi||\eta|} - \frac{(\xi\cdot\eta)\xi}{|\xi|^3|\eta|},
        \\
        &
        \partial_s
        \widetilde{\Gamma}(\lambda,c,s)
        =
        \int_0^\lambda
        \frac{cr+1-2c^2}{(r^2-2cr+1)^{3/2}}
        \, dr,
        \qquad
        \nabla_\xi s = \frac{\eta^\perp}{|\xi||\eta|} - \frac{(\xi\cdot\eta^\perp)\xi}{|\xi|^3|\eta|},
    \end{align}
    we obtain the second estimate of (3). 
    We also obtain the second estimate of (4) as $\gamma(\xi,\eta)=-2s(\xi,\eta)c(\xi,\eta)$.
    The first estimate of (4) follows from
    \begin{align}
        &
        \nabla_\xi \sp{\frac{\eta^\perp \cdot (\xi-\eta)}{|\eta||\xi-\eta|}}
        =
        \frac{\eta^\perp}{|\eta||\xi-\eta|}-\frac{(\eta^\perp \cdot (\xi-\eta))(\xi-\eta)}{|\eta||\xi-\eta|^3},
        \\
        &
        \nabla_\xi \sp{\frac{|\xi-\eta|^2-|\eta|^2}{|\xi|^2}}
        =
        \frac{2(\xi-\eta)}{|\xi|^2}
        -
        \frac{|\xi-\eta|^2-|\eta|^2}{|\xi|^2}
        \frac{2\xi}{|\xi|^2}.
    \end{align}
    Thus, we complete the proof.
\end{proof}
In the view of the third assertion of Lemma \ref{lemm:Gamma}, 
we consider the operator $\mathcal{H}_j$ with $\Gamma_\xi(\xi-\eta,\eta)$ replaced by $\gamma(\xi,\eta)$:
\begin{align}
    \widetilde{\mathcal{H}}_j[\psi](\xi)
    :={}&
    \frac{1}{2\pi|\xi|}
    \int_{\R^2}
    \gamma(\xi,\eta)V_j(\eta)\psi(\eta)\, d\eta
    \\
    ={}&
    \frac{\cos (2\theta_\xi)}{2\pi|\xi|}
    \int_{\R^2}
    \sin (2\theta_\eta)V_j(\eta)\psi(\eta)\, d\eta
    -
    \frac{\sin (2\theta_\xi)}{2\pi|\xi|}
    \int_{\R^2}
    \cos (2\theta_\eta)V_j(\eta)\psi(\eta)\, d\eta
    \\
    ={}&
    \frac{\cos (2\theta_\xi)}{2\pi|\xi|}
    \int_{\R^2}
    \sin (2\theta_\eta)V_j(\eta)\psi(\eta)\, d\eta.
    \label{df:tH}
\end{align}
Here, we have used 
\begin{align}
    \int_{\R^2}
    \cos (2\theta_\eta)V_j(\eta)\psi(\eta)\, d\eta
    =
    0,
\end{align}
which is implied by the change of the variables $(\eta_1,\eta_2) \mapsto (-\eta_1,\eta_2)$.
We define
\begin{align}
    \kappa \psi
    =\{\kappa_j[\psi]\}_{j \in \Z},
    \qquad
    \kappa_j[\psi]
    :=
    \frac{1}{2\pi}
    \int_{\R^2}
    \sin (2\theta_\eta)
    V_j(\eta)\psi(\eta)\, d\eta.
\end{align}
\begin{lemm}\label{lemm:k}
    For $j,k \in \Z$ and $\psi \in L^2_{\rm sym}$, it holds 
    \begin{align}
    \chi_k(\xi)\widetilde{\mathcal{H}}_j[\psi](\xi)
    =
    W_k(\xi)\kappa_j[\psi],
    \qquad
    \kappa_j[W_k]=\delta_{j,k}.
\end{align}
Moreover, there exists a positive constant $C$ such that 
\begin{align}
    \n{\kappa \psi}_{\ell^q} \leq C\| \psi \|_{X^q}
\end{align}
for all $1 \leq q \leq 2$ and $\psi \in X^q_{\rm sym}$.
\end{lemm}
\begin{proof}
    The first one immediately follows from \eqref{df:tH}.
    The second one is proved by 
    \begin{align}
        \kappa_j[W_k]
        &
        =
        \frac{\delta_{j,k}}{4\pi}
        \int_{\R^2}
        \rho_j(\eta)\frac{\sin^2 (4\theta_\eta)}{|\eta|^2}\, d\eta
        \\
        &
        =
        \frac{\delta_{j,k}}{4\pi}
        \int_0^\infty \frac{\rho(2^{-j}r)}{r}\, dr \int_{-\pi}^\pi \sin^2 (4\theta)\, d\theta
        =
        \delta_{j,k},
    \end{align}
    which is implied by $\rho_j(\eta)\chi_k(\eta)=\delta_{j,k}\rho_j(\eta)$.
    For the boundedness of $\kappa$, we see that
    \begin{align}
        |\kappa_j[\psi]|
        =
        |\kappa_j[\chi_j\psi]|
        \leq 
        \n{V_j}_{L^2}\n{\chi_j\psi}_{L^2}
        =
        \n{V_0}_{L^2}\n{\chi_j\psi}_{L^2}.
    \end{align}
    Taking $\ell^q(\Z)$-norm, we complete the proof.
\end{proof}
Now, we define the principal part of $\mathcal{L}_N$ as follows.
\begin{df}
    For $1 < q \leq 2$ and $N \in \N$, we define 
    \begin{align}
        &
        \widetilde{\mathcal{L}}_N\psi
        :=
        \psi - \mathcal{W}H_N\kappa\psi, 
        \qquad 
        \psi 
        \in 
        Y^q
        :=
        \Mp{
        \psi \in X^q_{\rm sym}\ ;\ \kappa\psi \in \ell_{\nu_N}^q(\Z)
        }.
    \end{align}
\end{df}
\begin{lemm}
    For $1 < q \leq 2$, $\widetilde{\mathcal{L}}_N$ is a bounded linear operator from $Y^q$ to $X^q_{\rm sym}$.
    Moreover, it holds 
    \begin{align}
        \operatorname{ker}\widetilde{\mathcal{L}}_N=\operatorname{span}\{\mathcal{W}\omega_N\}.
    \end{align}
\end{lemm}
\begin{proof}
    Since the operators $\kappa:Y^q \to \ell_{\nu_N}^q(\Z)$, $H_N:\ell_{\nu_N}^q(\Z)\to \ell^q(\Z)$, and $\mathcal{W}:\ell^q(\Z) \to X^q_{\rm sym}$ are all bounded, we see that $\widetilde{\mathcal{L}}_N$ is bounded from $Y^q$ to $X^q_{\rm sym}$.
    Next, we focus on the kernel of $\widetilde{\mathcal{L}}_N$. It follows from Lemma \ref{lemm:k} that 
    \begin{align}
        \kappa_j[\mathcal{W}p]
        =
        \sum_{k \in \Z}\kappa_j[W_k]p_k=p_j,
    \end{align}
    that is, $\kappa \mathcal{W} = \Id$.
    Thus, we have
    \begin{align}
        \widetilde{\mathcal{L}}_N\mathcal{W}\omega_N
        =\mathcal{W}\omega_N - \mathcal{W} H_N \omega_N =\mathcal{W}L_N\omega_N = 0,
    \end{align}
    which implies that it suffices to show $\operatorname{ker}\widetilde{\mathcal{L}}_N\subset \operatorname{span}\{\mathcal{W}\omega_N\}$ in the rest of the proof. 
    Let $\psi \in \operatorname{ker}\widetilde{\mathcal{L}}_N$.
    Since $\psi = \mathcal{W}H_N\kappa\psi$ and $\kappa \mathcal{W} = \Id$, we have
    \begin{align}
        \kappa \psi = \kappa \mathcal{W}H_N\kappa \psi = H_N \kappa \psi.
    \end{align}
    Hence, we see that $\kappa \psi \in \operatorname{ker}L_N$ and Lemma \ref{lemm:Ker_L_N} yields $\kappa \psi = c \omega_N$ for some constant $c \in \R$.
    Then, we have
    \begin{align}
        \psi = \mathcal{W}H_N\kappa \psi = c \mathcal{W}H_N\omega_N= c \mathcal{W}\omega_N \in \operatorname{span}\{ \mathcal{W}\omega_N\}.
    \end{align}
    This completes the proof.
\end{proof}
Let us next find a right inverse of $\widetilde{\mathcal{L}}_N$.
For fixed $\psi \in X^q_{\rm sym}$, we shall solve $\widetilde{\mathcal{L}}_N\phi=\psi$.
By $\phi=\mathcal{W}\kappa \phi + (\Id - \mathcal{W}\kappa) \phi$ 
and $\psi=\mathcal{W}\kappa \psi + (\Id - \mathcal{W}\kappa) \psi$, 
we see that 
$\mathcal{W}L_N\kappa \phi + (\Id - \mathcal{W}\kappa)\phi = \mathcal{W}\kappa \psi + (\Id - \mathcal{W}\kappa) \psi$, which yields
\begin{align}\label{phi-psi}
    \mathcal{W}(L_N\kappa \phi - \kappa \psi) = (\Id - \mathcal{W}\kappa)(\psi-\phi).
\end{align}
Hence, applying $\kappa$ to the above identity, we have by $\kappa \mathcal{W} = \Id$ that $L_N\kappa \phi - \kappa \psi = 0$, which is equivalent to $L_N(\kappa \phi - R_N\kappa \psi) = 0$.
Then, we see by Lemma \ref{lemm:Ker_L_N} that $\kappa \phi = c \omega_N+R_N\kappa\psi$ for some $c \in \R$. Substituting this into the right-hand side of \eqref{phi-psi}, we have $(\Id - \mathcal{W}\kappa)\phi=(\Id - \mathcal{W}\kappa)\psi$.
Therefore, it holds 
\begin{align}
    \phi 
    = 
    \mathcal{W}\kappa\phi + (\Id -\mathcal{W}\kappa)\phi
    = 
    c \mathcal{W}\omega_N + \mathcal{W}R_N\kappa\psi + (\Id -\mathcal{W}\kappa)\psi.
\end{align}
From this observation, we immediately get the following lemma.
\begin{lemm}
    Let $1 < q \leq 2$ and $N \in \N$. Let
    \begin{align}
        \mathcal{R}_N\psi:=\mathcal{W}R_N\kappa \psi + (\Id-\mathcal{W}\kappa)\psi,
        \qquad
        \psi \in X^q_{\rm sym}.
    \end{align}
    Then, $\mathcal{R}_N$ is a bounded linear operator from $X^q_{\rm sym}$ to $Y^q$ and it holds $\widetilde{\mathcal{L}}_N\mathcal{R}_N=\Id$.
\end{lemm}
\subsection{Error estimates for the Oseen operator}
In this subsection, we focus on the following remainder terms:
\begin{align}
    \mathcal{E}_N[\psi](\xi)
    &:=
    -\sum_{j,k,\ell \in \Z}\nu_{N,k}E_j^{k,\ell}[\psi](\xi),
    \\
    E_j^{k,\ell}[\psi](\xi)
    &:=\chi_j(\xi)\mathcal{H}_k[\chi_\ell\psi](\xi) - \bb{1}_{\{j<k=\ell\}}(j,k,\ell)\chi_j(\xi)\widetilde{\mathcal{H}}_k[\psi](\xi).
\end{align}
Then, it holds 
\begin{align}
    \mathcal{L}_N\psi=\widetilde{\mathcal{L}}_N\psi + \mathcal{E}_N\psi, \qquad \psi \in Y^q,
\end{align}
which follows from
\begin{align}
    \mathcal{E}_N[\psi]
    ={}&
    -\sum_{j\in \Z}\chi_j\sum_{k \in \Z} \nu_{N,k}\mathcal{H}_k\lp{\sum_{\ell \in \Z}\chi_\ell\psi}
    +
    \sum_{j \in \Z}
    \chi_j
    \sum_{k>j}\nu_{N,k}
    \widetilde{\mathcal{H}}_k[\psi]
    \\
    ={}&
    \mathcal{O}_N[\psi]
    +
    \mathcal{W}H_N\kappa\psi
    \\
    ={}&
    \mathcal{L}_N[\psi]
    -
    \widetilde{\mathcal{L}}_N[\psi].
\end{align}
The aim of this subsection is to prove the following proposition.
\begin{prop}\label{prop:E_N}
    Let $1 \leq q \leq 2$ and $N \in \N$.
    Then, the linear operator $\mathcal{E}_N$ is bounded on $X^q_{\rm sym}$.
    Moreover, there exists a positive constant $C$ independent of $q$ and $N$ such that 
    \begin{align}
        \n{\mathcal{E}_N[\psi]}_{X^q} \leq \frac{C}{N} \n{\psi}_{X^q}
    \end{align}
    for all $\psi \in X^q_{\rm sym}$.
\end{prop}
Before starting the proof of Proposition \ref{prop:E_N}, we prepare some lemmas.
\begin{lemm}\label{lemm:E-supp}
    Let $A=A_1 \sqcup A_2$, where
    \begin{align}
        A_1&:=\Mp{(j,k,\ell)\in \Z^3\ ;\ k=\ell,\ j \leq k+2},
        \\
        A_2&:=\Mp{(j,k,\ell)\in \Z^3\ ;\ k \neq \ell,\ \max\{k,\ell\}-5 \leq j \leq \max\{k,\ell\}+2}.
    \end{align}
    Then, it holds 
    \begin{align}
        \mathcal{E}_N[\psi]
        &
        =
        -
        \sum_{(j,k,\ell)\in A}
        \nu_{N,k}E_j^{k,\ell}[\psi].
    \end{align}
\end{lemm}
\begin{proof}
Assume that $\xi \in \operatorname{supp} E_j^{k,\ell}[\psi]$.
We first consider the case of $k=\ell$ with $j \geq k$.
Then, since 
\begin{align}\label{E_neq_0}
    \frac{\chi_j(\xi)}{2\pi|\xi|}
    \int_{\R^2}
    \Gamma_\xi(\xi-\eta,\eta)
    V_k(\xi-\eta)
    \chi_\ell (\eta)\psi(\eta)\, d\eta \neq 0,
\end{align}
it holds 
\begin{align}
    2^{j-\frac{3}{4}} \leq |\xi| \leq 2^{j+\frac{3}{4}},
    \qquad
    2^{k-\frac{1}{8}} \leq | \xi-\eta | \leq 2^{k+\frac{1}{8}},
    \qquad
    2^{k-\frac{3}{4}} \leq | \eta | \leq 2^{k + \frac{3}{4}}.
\end{align}
Then, we have 
\begin{align}
    2^{j-\frac{3}{4}}
    \leq
    |\xi|
    \leq 
    |\xi-\eta|
    +
    |\eta|
    \leq 
    2^{k+\frac{1}{8}}
    +
    2^{k+\frac{3}{4}}
    \leq
    2^{k+\frac{7}{4}},
\end{align}
which implies $(j,k,\ell) \in A_1$.
Next, we consider the case of $k \neq \ell$.
By \eqref{E_neq_0}, we see that 
\begin{align}
    2^{j-\frac{3}{4}} \leq |\xi| \leq 2^{j+\frac{3}{4}},
    \qquad
    2^{k-\frac{1}{8}} \leq | \xi-\eta | \leq 2^{k+\frac{1}{8}},
    \qquad
    2^{\ell-\frac{3}{4}} \leq | \eta | \leq 2^{\ell + \frac{3}{4}}.
\end{align}
From this, it follows for $k>\ell$ that 
\begin{align}
    &
    2^{j+\frac{3}{4}} \geq |\xi| \geq |\xi-\eta| - |\eta| \geq 2^{k-\frac{1}{8}} - 2^{\ell + \frac{3}{4}}
    \geq 2^{k-\frac{1}{8}} - 2^{k-\frac{1}{4}}
    >
    2^{k-\frac{1}{8}-4},
    \\
    &
    2^{j-\frac{3}{4}}
    \leq
    |\xi| 
    \leq 
    |\xi-\eta| + |\eta|
    \leq 
    2^{k+\frac{1}{8}} + 2^{\ell + \frac{3}{4}}
    \leq 
    2^{k+\frac{1}{8}} + 2^{k - \frac{1}{4}}
    <
    2^{k+\frac{1}{8}+1},
\end{align}
where we have used $2^{-\frac{1}{8}}<15/16$.
When $\ell>k$, we see that 
\begin{align}
    &
    2^{j+\frac{3}{4}} \geq |\xi| \geq |\eta| - |\xi-\eta| \geq  2^{\ell - \frac{3}{4}} - 2^{k + \frac{1}{8}}
    \geq 2^{\ell-\frac{3}{4}} - 2^{\ell-\frac{7}{8}}
    >
    2^{\ell-\frac{3}{4}-4},
    \\
    &
    2^{j-\frac{3}{4}}
    \leq
    |\xi| 
    \leq 
    |\xi-\eta| + |\eta|
    \leq 
    2^{k+\frac{1}{8}} + 2^{\ell + \frac{3}{4}}
    \leq 
    2^{\ell-\frac{7}{8}} + 2^{\ell + \frac{3}{4}}
    <
    2^{\ell+\frac{3}{4}+1},
\end{align}
Thus, we have $(j,k,\ell) \in A_2$ and complete the proof.
\end{proof}

\begin{lemm}\label{lemm:E-est}
    There exists a positive constant $C$ such that 
    \begin{align}
        \n{E_j^{k,\ell}[\psi]}_{L^2} \leq C\delta_{j,k,\ell}\| \chi_\ell \psi \|_{L^2}
    \end{align}
    for all $(j,k,\ell) \in A$ and $\psi \in L^2_{\rm loc}(\R^2\setminus\{0\})$,
    where 
    \begin{align}
        \delta_{j,k,\ell}
        :=
        \begin{cases}
            2^{j-k} & (j,k,k) \in A_1, \\
            2^{-|k-\ell|}  &(j,k,\ell) \in A_2.
        \end{cases}
    \end{align}
\end{lemm}
\begin{proof}
We first consider the case of $(j,k,k) \in A_1$ with $j \leq k-3$.
Since
\begin{align}\label{Ekk}
    E_j^{k,k}[\psi](\xi)
    ={}&
    \frac{\chi_j(\xi)}{2\pi|\xi|}
    \int_{\R^2}
    \sp{\Gamma_\xi(\xi-\eta,\eta)-\gamma(\xi,\eta)}
    V_k(\xi-\eta)
    \chi_k(\eta)\psi(\eta)
    \, d\eta
    \\
    &
    +
    \frac{\chi_j(\xi)}{2\pi|\xi|}
    \int_{\R^2}
    \gamma(\xi,\eta)
    \sp{V_k(\xi-\eta)-V_k(-\eta)}
    \chi_k(\eta)\psi(\eta)
    \, d\eta
\end{align}
and $|\xi|/|\eta| \leq 2^{j+\frac{3}{4}}/2^{k-\frac{3}{4}}=2^{j-k+\frac{3}{2}}\leq 2^{-\frac{3}{2}}<1$, 
we see by Lemma \ref{lemm:Gamma} that
\begin{align}
    \abso{E_j^{k,k}[\psi](\xi)}
    \leq{}& 
    \frac{\chi_j(\xi)}{2\pi|\xi|}
    \int_{\R^2}
    \abso{\Gamma_\xi(\xi-\eta,\eta)-\gamma(\xi,\eta)}
    |V_k(\xi-\eta)|
    |\chi_k(\eta)\psi(\eta)|
    \, d\eta
    \\
    &
    +
    \frac{\chi_j(\xi)}{2\pi|\xi|}
    \int_{\R^2}
    |V_k(\xi-\eta)-V_k(-\eta)|
    |\chi_k(\eta)\psi(\eta)|
    \, d\eta
    \\
    \leq{}& 
    C
    \frac{\chi_j(\xi)}{|\xi|}
    \int_{\R^2}
    \frac{|\xi|}{|\eta|}
    |V_k(\xi-\eta)|
    |\chi_k(\eta)\psi(\eta)|
    \, d\eta
    \\
    &
    +
    C
    \frac{\chi_j(\xi)}{|\xi|}
    \n{V_k(\cdot-\xi)-V_k(\cdot)}_{L^2}
    \n{\chi_k\psi}_{L^2}
    \\
    \leq{}& 
    C2^{-k}\chi_j(\xi)\n{V_k}_{L^2}\n{\chi_k\psi}_{L^2}
    +
    C\chi_j(\xi)\n{\nabla V_k}_{L^2}\n{\chi_k\psi}_{L^2}
    \\
    \leq{}& 
    C2^{-k}\chi_j(\xi)\n{\chi_k\psi}_{L^2}.
\end{align}
Taking $L^2$-norm and using $\| \chi_j \|_{L^2} \leq C 2^j$, we obtain
\begin{align}
    \n{E_j^{k,k}[\psi]}_{L^2}\leq C2^{j-k}\n{\chi_k\psi}_{L^2}.
\end{align}
For the case of $(j,k,k) \in A_1$ with $k-2 \leq j \leq k+2$, 
we see that 
\begin{align}
    &
    \abso{\chi_j(\xi)\mathcal{H}_k[\chi_k\psi](\xi)}
    +
    |\chi_j(\xi)\widetilde{\mathcal{H}}_k[\psi](\xi)|
    \\
    &\quad 
    \leq 
    C\frac{\chi_j(\xi)}{|\xi|}\int_{\R^2}|V_k(\xi-\eta)||\chi_\ell(\eta) \psi(\eta)|\, d\eta
    +
    C\frac{\chi_j(\xi)}{|\xi|}\int_{\R^2}|V_k(-\eta)||\chi_\ell(\eta) \psi(\eta)|\, d\eta
    \\
    &\quad 
    \leq 
    C\frac{\chi_j(\xi)}{|\xi|}\n{V_k}_{L^2}\n{\chi_k\psi}_{L^2}
    \leq 
    C\frac{\chi_j(\xi)}{|\xi|}\n{\chi_k\psi}_{L^2},
\end{align}
which implies
\begin{align}
    \n{E_j^{k,k}[\psi]}_{L^2}\leq C\n{\chi_k\psi}_{L^2}.
\end{align}

Next, we consider the case of $(j,k,\ell) \in A_2$.
As we see that 
\begin{align}
    \abso{E_j^{k,\ell}[\psi](\xi)}
    \leq{}& 
    C2^{-\max\{k,\ell\}}
    \int_{\R^2}
    |V_k(\xi-\eta)|
    |\chi_\ell(\eta)\psi(\eta)|
    \, d\eta,
\end{align}
the Hausdorff--Young inequality yields
\begin{align}
    \n{E_j^{k,\ell}[\psi]}_{L^2}
    &\leq 
    \begin{cases}
        C2^{-k}\n{V_k}_{L^2}\n{\chi_\ell\psi}_{L^1} & (\ell<k), \\
        C2^{-\ell}\n{V_k}_{L^1}\n{\chi_\ell\psi}_{L^2} & (\ell>k)
    \end{cases}
    \\
    &\leq 
    \begin{cases}
        C2^{-k}|\supp \chi_\ell|^{\frac{1}{2}}
        \n{\chi_\ell\psi}_{L^2} & (\ell<k), \\
        C2^{-\ell}2^k\n{\chi_\ell \psi}_{L^2} & (\ell>k)
    \end{cases}
    \\
    &\leq 
    C2^{-|k-\ell|}\n{\chi_\ell \psi}_{L^2}.
\end{align}
Thus, we complete the proof.
\end{proof}
\begin{lemm}\label{lemm:nabla-E-est}
    There exists a positive constant $C$ such that 
    \begin{align}
        2^j
        \n{\nabla E_j^{k,\ell}[\psi]}_{L^2} \leq C\| \chi_\ell \psi \|_{L^2}
    \end{align}
    for all $(j,k,\ell) \in A$ and $\psi \in L^2_{\rm loc}(\R^2\setminus\{0\})$.
\end{lemm}
\begin{proof}
We first consider the case of $(j,k,k) \in A_1$ with $j \leq k-3$.
Differentiating \eqref{Ekk}, we have 
\begin{align}\label{DEkk}
    \nabla E_j^{k,k}[\psi](\xi)
    ={}&
    \nabla \sp{\frac{\chi_j(\xi)}{2\pi|\xi|}}
    \int_{\R^2}
    \sp{\Gamma_\xi(\xi-\eta,\eta)-\gamma(\xi,\eta)}
    V_k(\xi-\eta)
    \chi_k(\eta)\psi(\eta)
    \, d\eta
    \\
    &
    +
    \frac{\chi_j(\xi)}{2\pi|\xi|}
    \int_{\R^2}
    \nabla_\xi\sp{\Gamma_\xi(\xi-\eta,\eta)-\gamma(\xi,\eta)}
    V_k(\xi-\eta)
    \chi_k(\eta)\psi(\eta)
    \, d\eta
    \\
    &
    +
    \frac{\chi_j(\xi)}{2\pi|\xi|}
    \int_{\R^2}
    \Gamma_\xi(\xi-\eta,\eta)
    (\nabla V_k)(\xi-\eta)
    \chi_k(\eta)\psi(\eta)
    \, d\eta
    \\
    &
    +
    \nabla
    \sp{\frac{\chi_j(\xi)}{2\pi|\xi|}}
    \int_{\R^2}
    \gamma(\xi,\eta)
    \sp{V_k(\xi-\eta)-V_k(-\eta)}
    \chi_k(\eta)\psi(\eta)
    \, d\eta
    \\
    &
    +
    \frac{\chi_j(\xi)}{2\pi|\xi|}
    \int_{\R^2}
    \nabla_\xi\gamma(\xi,\eta)
    \sp{V_k(\xi-\eta)-V_k(-\eta)}
    \chi_k(\eta)\psi(\eta)
    \, d\eta
    \\
    =:{}&
    E_{j,1}^{k,k}[\psi](\xi)+E_{j,2}^{k,k}[\psi](\xi)+E_{j,3}^{k,k}[\psi](\xi)+E_{j,4}^{k,k}[\psi](\xi)+E_{j,5}^{k,k}[\psi](\xi).
\end{align}
Let us consider the estimate of $E_{j,1}^{k,k}[\psi]$.
It follows from Lemma \ref{lemm:Gamma} and 
\begin{align}\label{D-chi}
    \abso{\nabla \sp{\frac{\chi_j(\xi)}{|\xi|}}}
    =
    \abso{\frac{2^{-j}(\nabla \chi_0)(2^{-j}\xi)}{|\xi|}
    -
    \frac{\chi_0(2^{-j}\xi)\xi}{|\xi|^3}}
    \leq 
    C2^{-2j}\bb{1}_{\supp \chi_j}(\xi)
\end{align}
that 
\begin{align}
    \abso{E_{j,1}^{k,k}[\psi](\xi)}
    &\leq 
    \abso{\nabla \sp{\frac{\chi_j(\xi)}{|\xi|}}}
    \int_{\R^2}
    \abso{\Gamma_\xi(\xi-\eta,\eta)-\gamma(\xi,\eta)}
    |V_k(\xi-\eta)|
    |\chi_k(\eta)\psi(\eta)|
    \, d\eta
    \\
    &
    \leq 
    C2^{-2j}
    \int_{\R^2}
    \frac{|\xi|}{|\eta|}
    |V_k(\xi-\eta)|
    |\chi_k(\eta)\psi(\eta)|
    \, d\eta
    \\
    &
    \leq 
    C2^{-j-k}
    \int_{\R^2}
    |V_k(\xi-\eta)|
    |\chi_k(\eta)\psi(\eta)|
    \, d\eta.
\end{align}
Thus, the Hausdorff--Young inequality implies 
\begin{align}\label{est:E1}
    2^j
    \n{E_{j,1}^{k,k}[\psi]}_{L^2}
    \leq 
    C
    2^{-k}
    \n{V_k}_{L^1}
    \n{\chi_k \psi}_{L^2}
    \leq 
    C
    \n{\chi_k \psi}_{L^2}.
\end{align}
For the estimate of $E_{j,2}^{k,k}[\psi]$,
we have from Lemma \ref{lemm:Gamma} that 
\begin{align}
    \abso{E_{j,2}^{k,k}[\psi](\xi)}
    &
    \leq 
    \frac{\chi_j(\xi)}{2\pi|\xi|}
    \int_{\R^2}
    |\nabla_\xi\sp{\Gamma_\xi(\xi-\eta,\eta)-\gamma(\xi,\eta)}|
    |V_k(\xi-\eta)|
    |\chi_k(\eta)\psi(\eta)|
    \, d\eta
    \\
    &
    \leq 
    C
    \frac{\chi_j(\xi)}{|\xi|}
    \int_{\R^2}
    \frac{1}{|\eta|}
    |V_k(\xi-\eta)|
    |\chi_k(\eta)\psi(\eta)|
    \, d\eta
    \\
    &
    \leq 
    C
    2^{-j-k}
    \chi_j(\xi)
    \|V_k\|_{L^2}
    \|\chi_k\psi\|_{L^2},
\end{align}
which implies 
\begin{align}\label{est:E2}
    2^j\n{E_{j,2}^{k,k}[\psi]}_{L^2}
    \leq 
    C2^{j-k}
    \|\chi_k\psi\|_{L^2}
    \leq 
    C
    \|\chi_k\psi\|_{L^2}.
\end{align}
On the estimate of $E_{j,3}^{k,k}[\psi]$, as 
\begin{align}
    \abso{E_{j,3}^{k,k}[\psi](\xi)}
    &
    \leq 
    C2^{-j}\chi_j(\xi)
    \int_{\R^2}|\nabla V_k(\xi-\eta)||\chi_k(\eta)\psi(\eta)|
    \, d\eta
    \\
    &
    \leq 
    C2^{-j}\chi_j(\xi)
    \|\nabla V_k\|_{L^2}
    \|\chi_k\psi\|_{L^2}
    \leq 
    C2^{-j-k}\chi_j(\xi)
    \|\chi_k\psi\|_{L^2},
\end{align}
we see that 
\begin{align}\label{est:E3}
    2^j
    \n{E_{j,3}^{k,k}[\psi]}_{L^2}
    \leq C2^{j-k}\n{\chi_k \psi}_{L^2}
    \leq C\n{\chi_k \psi}_{L^2}.
\end{align}
For $E_{j,4}^{k,k}[\psi]$, we have 
\begin{align}
    \abso{E_{j,4}^{k,k}[\psi](\xi)}
    &
    \leq 
    C2^{-2j}
    \bb{1}_{\supp \chi_j}(\xi)
    \| V_k(\cdot - \xi) - V_k(\cdot) \|_{L^2}
    \| \chi_k \psi \|_{L^2}
    \\
    &
    \leq 
    C2^{-2j}
    \bb{1}_{\supp \chi_j}(\xi)
    |\xi|
    \| \nabla V_k \|_{L^2}
    \| \chi_k \psi \|_{L^2} 
    \leq 
    C2^{-j-k}
    \bb{1}_{\supp \chi_j}(\xi)
    \| \chi_k \psi \|_{L^2},
\end{align}
which yields 
\begin{align}\label{est:E4}
    2^j\n{E_{j,4}^{k,k}[\psi]}_{L^2}
    \leq 
    C2^{j-k}\| \chi_k \psi \|_{L^2}
    \leq 
    C\| \chi_k \psi \|_{L^2}.
\end{align}
We consider the estimate of $E_{j,5}^{k,k}[\psi]$.
By Lemma \ref{lemm:Gamma}, we have 
\begin{align}
    \abso{E_{j,5}^{k,k}[\psi](\xi)}
    &
    \leq 
    C\frac{\chi_j(\xi)}{|\xi|^2}
    \n{V_k(\cdot-\xi)-V_k(\cdot)}_{L^2}
    \n{\chi_k\psi}_{L^2}
    \\
    &
    \leq 
    C
    \frac{\chi_j(\xi)}{|\xi|}
    \n{\nabla V_k}_{L^2}
    \n{\chi_k\psi}_{L^2}
    \\
    &\leq 
    C
    2^{-j-k}
    \chi_j(\xi)
    \n{\chi_k\psi}_{L^2}.
\end{align}
Thus, we have 
\begin{align}\label{est:E5}
    2^j
    \n{E_{j,5}^{k,k}[\psi]}_{L^2}
    \leq 
    C2^{j-k}
    \n{\chi_k\psi}_{L^2}
    \leq 
    C\n{\chi_k\psi}_{L^2}.
\end{align}
Combining \eqref{est:E1}, \eqref{est:E2}, \eqref{est:E3}, \eqref{est:E4}, and \eqref{est:E5}, we obtain 
\begin{align}
    2^j
    \n{ \nabla E_{j}^{k,k}[\psi]}_{L^2}
    \leq 
    C\n{\chi_k\psi}_{L^2}.
\end{align}

Next, we consider the case of $(j,k,k) \in A_1$ with $k-2 \leq j \leq k+2$.
We see that 
\begin{align}
    \nabla E_j^{k,k}[\psi](\xi)
    ={}&
    \nabla\sp{\frac{\chi_j(\xi)}{2\pi|\xi|}}
    \int_{\R^2}
    \Gamma_\xi (\xi-\eta,\eta)
    V_k(\xi-\eta)\chi_k(\eta)\psi(\eta)\, d\eta
    \\
    &
    +
    \frac{\chi_j(\xi)}{2\pi|\xi|}
    \int_{\R^2}
    \nabla_\xi (\Gamma_\xi (\xi-\eta,\eta))
    V_k(\xi-\eta)\chi_k(\eta)\psi(\eta)\, d\eta
    \\
    &
    +
    \frac{\chi_j(\xi)}{2\pi|\xi|}
    \int_{\R^2}
    \Gamma_\xi (\xi-\eta,\eta)
    \nabla_\xi V_k(\xi-\eta)\chi_k(\eta)\psi(\eta)\, d\eta
    \\
    &
    -
    \bb{1}_{j<k}(j,k)
    \nabla W_j(\xi)
    \kappa_k[\psi].
\end{align}
Since $|\nabla_\xi \Gamma_\xi(\xi-\eta,\eta)| \leq C2^{-j}$ holds by Lemma \ref{lemm:Gamma}, 
it holds
\begin{align}
    \abso{\nabla E_j^{k,k}[\psi](\xi)}
    \leq {}&
    C
    2^{-2j}
    \bb{1}_{\supp \chi_j}(\xi)
    \| V_k \|_{L^2}
    \| \chi_k\psi \|_{L^2}
    \\
    &
    +
    C\frac{\chi_j(\xi)}{|\xi|}
    \| \nabla V_k \|_{L^2}
    \| \chi_k \psi \|_{L^2}
    +
    C
    | \nabla W_j(\xi) |
    \| V_k \|_{L^2}
    \| \chi_k \psi \|_{L^2}.
\end{align}
Taking $L^2$-norm, we have 
\begin{align}
    2^j\n{\nabla E_j^{k,k}[\psi]}_{L^2}
    \leq 
    C\| \chi_k\psi \|_{L^2}.
\end{align}

Finally, we focus on the case of $(j,k,\ell) \in A_2$.
We see that 
\begin{align}
    \nabla E_j^{k,\ell}[\psi](\xi)
    ={}&
    \nabla\sp{\frac{\chi_j(\xi)}{2\pi|\xi|}}
    \int_{\R^2}
    \Gamma_\xi (\xi-\eta,\eta)
    V_k(\xi-\eta)\chi_\ell (\eta)\psi(\eta)\, d\eta
    \\
    &
    +
    \frac{\chi_j(\xi)}{2\pi|\xi|}
    \int_{\R^2}
    \nabla_\xi (\Gamma_\xi (\xi-\eta,\eta))
    V_k(\xi-\eta)\chi_\ell (\eta)\psi(\eta)\, d\eta
    \\
    &
    +
    \frac{\chi_j(\xi)}{2\pi|\xi|}
    \int_{\R^2}
    \Gamma_\xi (\xi-\eta,\eta)
    \nabla_\xi V_k(\xi-\eta)\chi_\ell (\eta)\psi(\eta)\, d\eta
\end{align}
It follows from Lemma \ref{lemm:Gamma} and $|\xi| \simeq 2^j \simeq 2^{\max \{k,\ell \}} \simeq \max \{ |\xi-\eta|,|\eta|\}$ that 
\begin{align}
    |\nabla_\xi \Gamma_\xi(\xi-\eta,\eta)| 
    &
    \leq C2^{k-2j}+C(2^{-j}+2^{-k}) \\
    &
    \leq C2^{-\min \{k,\ell\}}.
\end{align}
Thus, we have 
\begin{align}
    \abso{\nabla E_j^{k,\ell}[\psi](\xi)}
    \leq{}& 
    C2^{-2j}\bb{1}_{\supp \chi_j}(\xi)\n{V_k}_{L^2}\n{\chi_\ell\psi}_{L^2}
    \\
    &
    +
    C
    2^{-j}
    2^{-\min\{k,\ell\}}
    \int_{\R^2}
    |V_k(\xi-\eta)||\chi_\ell(\eta)\psi(\eta)|\, d\eta 
    \\
    &
    +
    C
    2^{-j}
    \chi_j(\xi)
    \int_{\R^2}
    |\nabla V_k(\xi-\eta)||\chi_\ell(\eta) \psi(\eta)| 
    \, d\eta.
\end{align}
Taking $L^2$-norm and using 
\begin{align}
    &
    \begin{aligned}
    \n{\int_{\R^2}
    |V_k(\xi-\eta)||\chi_\ell(\eta)\psi(\eta)|\, d\eta}_{L^2_\xi}
    &
    \leq
    \min \Mp{\n{V_k}_{L^1}\n{\chi_\ell \psi}_{L^2},\n{V_k}_{L^2}\n{\chi_\ell \psi}_{L^1}}
    \\
    &\leq 
    C\min \Mp{2^k\n{\chi_\ell \psi}_{L^2},|\supp \chi_\ell|^{\frac{1}{2}}\n{\chi_\ell \psi}_{L^2}}
    \\
    &\leq 
    C
    2^{\min \{ k,\ell \}}
    \n{\chi_\ell \psi}_{L^2},
    \end{aligned}
    \\
    &
    \begin{aligned}
    \n{\int_{\R^2}
    |\nabla V_k(\xi-\eta)||\chi_\ell(\eta) \psi(\eta)| 
    \, d\eta}_{L^2_\xi}
    \leq 
    C\n{ \nabla V_k}_{L^1}\n{\chi_\ell \psi}_{L^2}
    \leq 
    C\n{\chi_\ell \psi}_{L^2},
    \end{aligned}
\end{align}
we obtain 
\begin{align}
    2^j\n{\nabla E_j^{k,\ell}[\psi]}_{L^2}
    \leq 
    C\n{\chi_\ell \psi}_{L^2}.
\end{align}
Thus, we complete the proof.
\end{proof}
Now, let us prove Proposition \ref{prop:E_N}.
\begin{proof}[Proof of Proposition \ref{prop:E_N}]
We first consider the $L^2$-estimate of $\chi_j\mathcal{E}_N[\psi]$.
By $|\nu_{N,k}| \leq 1/N$, the almost orthogonality of $\{\chi_j \}_{j \in \Z}$, and 
Lemmas \ref{lemm:E-supp} and \ref{lemm:E-est}, it holds
\begin{align}
    \n{\chi_j \mathcal{E}_N[\psi]}_{L^2}
    \leq {}&
    \frac{C}{N}
    \sum_{|j'-j| \leq 1}
    \sum_{(j',k,\ell) \in A}
    \n{E_{j'}^{k,\ell}[\psi]}_{L^2}
    \\
    \leq {}&
    \frac{C}{N}
    \sum_{|j'-j| \leq 1}
    \sum_{j' \leq k+2}
    2^{j'-k}
    \n{\chi_k\psi}_{L^2}
    \\
    &
    +
    \frac{C}{N}
    \sum_{|j'-j| \leq 1}
    \sum_{k-5 \leq j' \leq k+2}
    \sum_{\ell < k}
    2^{\ell-k}
    \n{\chi_\ell\psi}_{L^2}
    \\
    &
    +
    \frac{C}{N}
    \sum_{|j'-j| \leq 1}
    \sum_{\ell-5 \leq j' \leq \ell+2}
    \sum_{k < \ell}
    2^{k - \ell}
    \n{\chi_\ell\psi}_{L^2}.
\end{align}
Taking $\ell^q$-norm with respect to $j \in \Z$ and using the Hausdorff--Young inequality for the discrete convolution, we have 
\begin{align}
    \sp{\sum_{j\in \Z}\n{\chi_j\mathcal{E}_N[\psi]}_{L^2}^q}^{\frac{1}{q}} \leq \frac{C}{N} \n{\psi}_{X^q}.
\end{align}

Next, we consider the estimate for the difference of $\chi_j\mathcal{E}_N[\psi]$.
It follows from Lemmas \ref{lemm:E-est} and \ref{lemm:nabla-E-est} that
\begin{align}
    \n{E_j^{k,\ell}[\psi](\cdot+\eta) - E_j^{k,\ell}[\psi](\cdot)}_{L^2}
    \leq{}&
    C
    \min 
    \Mp{2\n{E_j^{k,\ell}[\psi]}_{L^2}, |\eta| \n{\nabla E_j^{k,\ell}[\psi]}_{L^2}}
    \\
    \leq{}&
    C\n{\chi_\ell \psi}_{L^2}
    \min \Mp{\delta_{j,k,\ell},2^{-j}|\eta|}
    \\
    \leq{}&
    C
    \sp{\delta_{j,k,\ell}}^{1/2}
    \n{\chi_\ell \psi}_{L^2}
    \sp{\min \Mp{1,2^{-j}|\eta|}}^{1/2},
\end{align}
which implies 
\begin{align}
    &
    \n{
    \sp{\chi_jE_{j'}^{k,\ell}[\psi]}(\cdot+\eta)-\sp{\chi_jE_{j'}^{k,\ell}[\psi]}(\cdot)
    }_{L^2}
    \\
    &\quad 
    \leq 
    \n{\chi_j(\cdot+\eta)-\chi_j(\cdot)}_{L^\infty}
    \n{
    \chi_jE_{j'}^{k,\ell}[\psi]}_{L^2}
    \\
    &\qquad 
    +
    \n{E_{j'}^{k,\ell}[\psi](\cdot+\eta) - E_{j'}^{k,\ell}[\psi](\cdot)}_{L^2}
    \\
    &\quad 
    \leq 
    C
    \min \Mp{1,2^{-j}|\eta|}
    \delta_{j,k,\ell}
    \n{\chi_\ell \psi}_{L^2}
    \\
    &\qquad
    +
    C
    \sp{\delta_{j,k,\ell}}^{1/2}
    \n{\chi_\ell \psi}_{L^2}
    \sp{\min \Mp{1,2^{-j}|\eta|}}^{1/2}
    \\
    &\quad 
    \leq 
    C
    \sp{\delta_{j,k,\ell}}^{1/2}
    \n{\chi_\ell \psi}_{L^2}
    \sp{\min \Mp{1,2^{-j}|\eta|}}^{1/2}
\end{align}
for $|j'-j| \leq 1$.
Thus, we have via $|\nu_{N,k}| \leq 1/N$ that
\begin{align}
    &
    \n{\sp{\chi_j\mathcal{E}_N[\psi]}(\cdot+\eta) - \sp{\chi_j\mathcal{E}_N[\psi]}(\cdot)}_{L^2}
    \\
    &\quad
    \leq 
    \frac{C}{N}
    \sum_{|j'-j| \leq 1}
    \sum_{(j',k,\ell) \in A}
    \n{
    \sp{\chi_jE_{j'}^{k,\ell}[\psi]}(\cdot+\eta)-\sp{\chi_jE_{j'}^{k,\ell}[\psi]}(\cdot)}_{L^2}
    \\
    &\quad 
    \leq 
    \frac{C}{N}
    \sum_{|j'-j|\leq 1}
    \sum_{(j',k,\ell) \in A}
    \sp{\delta_{j',k,\ell}}^{1/2}
    \n{\chi_\ell \psi}_{L^2}
    \sp{\min \Mp{1,2^{-j}|\eta|}}^{1/2}
    \\
    &\quad
    \leq 
    \frac{C}{N}
    \sum_{|j'-j| \leq 1}
    \sum_{j' \leq k+2}
    2^{\frac{j'-k}{2}}
    \n{\chi_k\psi}_{L^2}
    \sp{\min \Mp{1,2^{-j}|\eta|}}^{1/2}
    \\
    &\qquad
    +
    \frac{C}{N}
    \sum_{|j'-j| \leq 1}
    \sum_{k-5 \leq j' \leq k+2}
    \sum_{\ell < k}
    2^{\frac{\ell-k}{2}}
    \n{\chi_\ell\psi}_{L^2}
    \sp{\min \Mp{1,2^{-j}|\eta|}}^{1/2}
    \\
    &\qquad
    +
    \frac{C}{N}
    \sum_{|j'-j| \leq 1}
    \sum_{\ell-5 \leq j' \leq \ell+2}
    \sum_{k < \ell}
    2^{\frac{k - \ell}{2}}
    \n{\chi_\ell\psi}_{L^2}
    \sp{\min \Mp{1,2^{-j}|\eta|}}^{1/2},
\end{align}
which implies 
\begin{align}
    \Mp{
    \sum_{j \in \Z}
    \sp{
    \sup_{\eta \in \R^2 \setminus \{0\}}
    \frac{\n{\sp{\chi_j\mathcal{E}_N[\psi]}(\cdot+\eta)-\chi_j\mathcal{E}_N[\psi](\cdot)}_{L^2}}{\sp{\min \Mp{1,2^{-j}|\eta|}}^{1/2}}
    }^q
    }^{\frac{1}{q}}
    \leq 
    \frac{C}{N}
    \n{\psi}_{X^q}.
\end{align}
Thus, we complete the proof.
\end{proof}
\subsection{Proof of Proposition \ref{prop:Oseen}}
Now, we are ready to prove Proposition \ref{prop:Oseen}.
It follows from Proposition \ref{prop:E_N} that 
\begin{align}
    \n{\mathcal{E}_N\mathcal{R}_N}_{X^q \to X^q}
    \leq 
    \n{\mathcal{E}_N}_{Y^q \to X^q}
    \n{\mathcal{R}_N}_{X^q \to Y^q}
    \leq
    \frac{C}{N}.
\end{align}
Hence, choosing $N$ sufficiently large so that $\n{\mathcal{E}_N\mathcal{R}_N}_{X^q \to X^q} \leq 1/2$, 
we see that the operator $\Id + \mathcal{E}_N\mathcal{R}_N$ is invertible on $X^q_{\rm sym}$ and we define
\begin{align}
    \varphi_N
    :=-\sp{\Id + \mathcal{E}_N\mathcal{R}_N}^{-1}\mathcal{E}_N\mathcal{W}\omega_N \in X^q_{\rm sym}, 
    \qquad
    \psi_N 
    := \mathcal{W}\omega_N + \mathcal{R}_N \varphi_N \in Y^q.
\end{align}
For the estimate of $\varphi_N$ and $\psi_N$, we see that 
\begin{align}
    &
    \n{\mathcal{W}\omega_N}_{X^q} \leq C \n{\omega_N}_{\ell^q}
    \leq CN^{\frac{1}{q}},
    \\
    &
    \n{\varphi_N}_{X^q}
    \leq \sum_{n=0}^\infty 
    \sp{\n{\mathcal{E}_N\mathcal{R}_N}_{X^q \to X^q}}^n \n{\mathcal{E}_N\mathcal{W}\omega_N}_{X^q}
    \leq 
    CN^{\frac{1}{q}-1},
    \\
    &
    \n{\psi_N}_{X^q} \leq C\n{\omega_N}_{\ell^q} + C \n{\varphi_N}_{X^q} \leq CN^{\frac{1}{q}}.
\end{align}
For the lower bound estimate of $\psi_N$, 
since we see that  
\begin{align}
    \abso{\chi_j(\xi)(\mathcal{W}\omega_N)(\xi)}
    &
    =
    \frac{|\cos (2\theta_\xi)|}{|\xi|}
    \chi_j(\xi)
    \sum_{m \in \Z}
    \omega_{N,m}\chi_m(\xi)
    \\
    &
    \geq 
    \omega_{N,j}
    \frac{|\cos (2\theta_\xi)|}{|\xi|}
    \chi_j(\xi)^2,
\end{align}
it holds
\begin{align}
    \n{\mathcal{W}\omega_N}_{X^q}
    &
    \geq 
    \sp{
    \sum_{j \in \Z}
    \n{\chi_jW_j\omega_{N,j}}_{L^2}^q
    }^{\frac{1}{q}}
    \\
    &
    \geq
    c
    \n{\omega_N}_{\ell^q}
    \geq cN^{\frac{1}{q}},
\end{align}
which implies
\begin{align}
    \n{\psi_N}_{X^q} 
    \geq
    \n{\mathcal{W}\omega_N}_{X^q} - \n{\mathcal{R}_N\varphi_N}_{X^q}
    \geq 
    cN^{\frac{1}{q}} - CN^{\frac{1}{q}-1} \geq c N^{\frac{1}{q}}
\end{align}
for sufficiently large $N$.
Since 
$\widetilde{\mathcal{L}}_N\mathcal{W}\omega_N=0$, 
$\widetilde{\mathcal{L}}_N \mathcal{R}_N = \Id$, 
and 
\begin{align}
    \mathcal{L}_N \psi_N 
    ={}
    \sp{\widetilde{\mathcal{L}}_N + \mathcal{E}_N}
    \lp{\mathcal{W}\omega_N + \mathcal{R}_N\varphi_N}
    =
    \mathcal{E}_N\mathcal{W}\omega_N + \sp{\Id + \mathcal{E}_N \mathcal{R}_N}\varphi_N
    =
    0,
\end{align}
$\psi_N$ is a solution to \eqref{eq:sc-O}.
It is easy to check that $w_N:=N^{-1} \mathcal{U}[\psi_N]$ is a solution to \eqref{eq:Oseen} and we have $\n{w_N}_{\dB_{2,q}^0} \leq C N^{-1} \n{\psi_N}_{X^q} \leq C N^{\frac{1}{q}-1}$. 
Thus, we complete the proof.

\section{Proof of Theorem \ref{thm:main}}\label{sec:pf}
We are in a position to present the proof of Theorem \ref{thm:main}.
It suffices to consider only the case of
$1<q\leq2$ since the result for $q>2$ follows from the
case $q=2$ and the embeddings
$\dB_{2,2}^s(\R^2)\hookrightarrow\dB_{2,q}^s(\R^2)$.
Let $v_N, w_N \in \dB_{2,q}^0(\R^2)$ be the vector fields constructed in the previous section. 
Let us define
\begin{align}
    u_N^+ := v_N + w_N, 
    \qquad
    u_N^- := v_N - w_N,
\end{align}
and 
\begin{align}
    f_N :={}& - \Delta v_N + \mathbb{P} \sp{(v_N \cdot \nabla)v_N + (w_N \cdot \nabla)w_N}
    \\
    ={}& - \Delta v_N + \mathbb{P} \div\sp{v_N \otimes v_N + w_N \otimes w_N}.
\end{align}
Note that the product terms $v_N \otimes v_N$ and $w_N \otimes w_N$ are well-defined since $v_N,w_N \in \dB_{2,q}^0(\R^2) \subset L^2(\R^2)$.
Then, we see that $u_N^\pm$ solve \eqref{eq:sNS} with $f=f_N$ and 
\begin{align}
    \n{u_N^\pm}_{\dB_{2,q}^0} \leq \n{v_N}_{\dB_{2,q}^0} + \n{w_N}_{\dB_{2,q}^0} \leq CN^{\frac{1}{q}-1},
    \qquad
    u_N^+ - u_N^- = 2w_N \not \equiv 0.
\end{align}
Therefore, the proof is reduced to verifying that $f_N$ is small in $\dB_{2,q}^{-2}(\R^2)$.
Although the nonlinear products are well-defined, $f_N \in \dB_{2,q}^{-2}(\R^2)$ is not obvious since the paraproduct is not bounded from $\dB_{2,q}^0(\R^2) \times \dB_{2,q}^0(\R^2)$ to $\dB_{2,q}^{-1}(\R^2)$.
To overcome this, we use a cancellation property for the nonlinear terms.
We first focus on the estimate of $\mathbb{P}\div(\Delta_k w_N \otimes \Delta_\ell w_N)$ and consider the matrix
\begin{align}
    \mathscr{F}\lp{\Delta_k w_N \otimes \Delta_\ell w_N}(\xi)
    ={}&
    \frac{1}{2\pi}
    \int_{\R^2}
    (\chi_k\widehat{w_N})(\xi-\eta) \otimes (\chi_\ell\widehat{w_N})(\eta)
    \, 
    d\eta
    \\
    ={}&
    \frac{-1}{2\pi N^2}
    \int_{\R^2}
    (\chi_k\psi_N)(\xi-\eta)(\chi_\ell\psi_N)(\eta)
    \frac{(\xi-\eta)^\perp}{|\xi-\eta|} \otimes \frac{\eta^\perp}{|\eta|}
    \, 
    d\eta.
\end{align}
Let us rewrite this matrix at $\xi=0$ as
\begin{align}
    \mathscr{F}\lp{\Delta_k w_N \otimes \Delta_\ell w_N}(0)
    &=
    \begin{pmatrix}
        a_{N,k,\ell} & -b_{N,k,\ell} \\
        -b_{N,k,\ell} & c_{N,k,\ell}
    \end{pmatrix},
    \\
    a_{N,k,\ell}
    &:=\frac{1}{2\pi N^2}
    \int_{\R^2}
    \frac{\eta_2^2}{|\eta|^2}
    (\chi_k\chi_\ell\psi_N^2)(\eta)
    \, 
    d\eta,
    \\
    b_{N,k,\ell}
    &:=\frac{1}{2\pi N^2}
    \int_{\R^2}
    \frac{\eta_1\eta_2}{|\eta|^2}
    (\chi_k\chi_\ell\psi_N^2)(\eta)
    \, 
    d\eta,
    \\
    c_{N,k,\ell}
    &:=\frac{1}{2\pi N^2}
    \int_{\R^2}
    \frac{\eta_1^2}{|\eta|^2}
    (\chi_k\chi_\ell\psi_N^2)(\eta)
    \, 
    d\eta.
\end{align}
By the change of variables $\eta \mapsto \eta^\perp$ and using the symmetric property of $\psi_N \in L^2_{\rm sym}$, we see that 
$a_{N,k,\ell}=c_{N,k,\ell}$ and $b_{N,k,\ell}=-b_{N,k,\ell}$.
Thus, it holds
\begin{align}
    \mathscr{F}\lp{\Delta_k w_N \otimes \Delta_\ell w_N}(0)
    =
    a_{N,k,\ell}\Id,
\end{align}
which yields 
\begin{align}
    i\widehat{\mathbb{P}}(\xi)
    \mathscr{F}[\Delta_k w_N \otimes \Delta_\ell w_N](0)\xi
    =
    ia_{N,k,\ell} 
    \widehat{\mathbb{P}}(\xi)
    \xi
    =
    0.
\end{align}
It follows from 
\begin{align}
    &
    \Delta_j
    \sum_{|k-\ell|\leq 2}
    (\Delta_kf \Delta_\ell g)
    =
    \Delta_j
    \sum_{k \geq j-4}
    \sum_{|k-\ell|\leq 2}
    (\Delta_kf \Delta_\ell g)
\end{align}
that
\begin{align}
    &
    \mathscr{F}
    \lp{\Delta_j
    \sum_{|k-\ell| \leq 2}
    \mathbb{P} \div (\Delta_kw_N \otimes \Delta_\ell w_N)
    }(\xi)\\
    &\quad
    =
    \chi_j(\xi)
    \sum_{k \geq j-4}
    \sum_{|k-\ell|\leq 2}
    i\widehat{\mathbb{P}}(\xi)
    \sp{
    \mathscr{F}[\Delta_kw_N \otimes \Delta_\ell w_N](\xi)
    -
    \mathscr{F}[\Delta_kw_N \otimes \Delta_\ell w_N](0)
    }\xi
    \\
    &
    \quad
    =
    \frac{i}{2\pi}
    \chi_j(\xi)
    \widehat{\mathbb{P}}(\xi)
    \sum_{k \geq j-4}
    \sum_{|k-\ell|\leq 2}
    \\
    &\qquad \quad 
    \left(\int_{\R^2}
    \sp{(\chi_k\widehat{w_N})(\xi-\eta)-(\chi_k\widehat{w_N})(-\eta)}
    \otimes 
    (\chi_\ell\widehat{w_N})(\eta)
    \, 
    d\eta \right)\xi.
\end{align}
We then have
\begin{align}
    &
    \abso{
    \mathscr{F}
    \lp{\Delta_j
    \sum_{|k-\ell| \leq 2}
    \mathbb{P} \div (\Delta_kw_N \otimes \Delta_\ell w_N)
    }(\xi)
    }
    \\
    &\quad
    \leq
    \chi_j(\xi)|\xi|
    \sum_{k \geq j-4}
    \sum_{|k-\ell|\leq 2}
    \int_{\R^2}
    \abso{(\chi_k\widehat{w_N})(\xi-\eta)-(\chi_k\widehat{w_N})(-\eta)}
    |(\chi_\ell\widehat{w_N})(\eta)|
    \, 
    d\eta
    \\
    &\quad 
    \leq 
    C
    \chi_j(\xi)|\xi|
    \sum_{k \geq j-4}
    \sum_{|k-\ell|\leq 2}
    \n{(\chi_k\widehat{w_N})(\cdot-\xi)-(\chi_k\widehat{w_N})(\cdot)}_{L^2}
    \n{\chi_\ell\widehat{w_N}}_{L^2}
    \\
    &\quad 
    \leq 
    C
    \chi_j(\xi)|\xi|
    \sum_{k \geq j-4}
    \sum_{|k-\ell|\leq 2}
    \sp{\min\{1,2^{-k}|\xi|\}}^{1/2}
    \\
    &\qquad \qquad \qquad \qquad \qquad \qquad \quad
    \sup_{\xi' \neq 0}
    \frac{\n{(\chi_k\widehat{w_N})(\cdot+\xi')-(\chi_k\widehat{w_N})(\cdot)}_{L^2}}{\sp{\min\{1,2^{-k}|\xi'|\}}^{1/2}}
    \n{\chi_\ell\widehat{w_N}}_{L^2},
\end{align}
which implies 
\begin{align}
    &
    2^{-2j}
    \n{\Delta_j
    \sum_{|k-\ell| \leq 2}
    \mathbb{P} \div (\Delta_kw_N \otimes \Delta_\ell w_N)}_{L^2}
    \\
    &\quad
    \leq 
    C
    \sum_{k \geq j-4}
    \sum_{|k-\ell|\leq 2}
    2^{\frac{j-k}{2}}
    \sup_{\xi' \neq 0}
    \frac{\n{(\chi_k\widehat{w_N})(\cdot+\xi')-(\chi_k\widehat{w_N})(\cdot)}_{L^2}}{\sp{\min\{1,2^{-k}|\xi'|\}}^{1/2}}
    \n{\chi_\ell\widehat{w_N}}_{L^2}.
\end{align}
Taking $\ell^q$-norm and using the Hausdorff--Young inequality for the discrete convolution, we obtain 
\begin{align}
    \n{
    \sum_{|k-\ell| \leq 2}
    \mathbb{P} \div (\Delta_kw_N \otimes \Delta_\ell w_N)}_{\dB_{2,q}^{-2}}
    \leq
    \frac{C}{N^2}
    \n{\psi_N}_{X^q}^2
    \leq 
    CN^{2(\frac{1}{q}-1)}.
\end{align}
By the standard paraproduct estimates (see \cite[Theorem 2.47, p.~87]{Bah-Che-Dan-11}),
\begin{align}
    &
    \n{
    \sum_{\ell \leq k - 3}
    \mathbb{P} \div (\Delta_kw_N \otimes \Delta_\ell w_N)}_{\dB_{2,q}^{-2}}
    \leq 
    C\n{w_N}_{\dB_{2,q}^0}^2
    \leq 
    CN^{2(\frac{1}{q}-1)},
    \\
    &\n{
    \sum_{k \leq \ell - 3}
    \mathbb{P} \div (\Delta_kw_N \otimes \Delta_\ell w_N)}_{\dB_{2,q}^{-2}}
    \leq 
    C\n{w_N}_{\dB_{2,q}^0}^2
    \leq 
    CN^{2(\frac{1}{q}-1)}.
\end{align}
Thus, we obtain 
\begin{align}
    \n{
    \mathbb{P} \div (w_N \otimes w_N)}_{\dB_{2,q}^{-2}}
    \leq 
    CN^{2(\frac{1}{q}-1)}.
\end{align}
Since $\| \phi_N \|_{X^q}\leq CN^{\frac{1}{q}-1}$ and $\phi_N \in \widetilde{X}^q_{\rm sym}$, the same argument as above yields 
\begin{align}
    \n{\mathbb{P} \div (v_N \otimes v_N)}_{\dB_{2,q}^{-2}}
    \leq 
    CN^{2(\frac{1}{q}-1)}.
\end{align}
Hence, we see that $f_N \in \dB_{2,q}^{-2}(\R^2)$ and 
\begin{align}
    \n{f_N}_{\dB_{2,q}^{-2}}
    \leq 
    C
    \n{v_N}_{\dB_{2,q}^0}
    +
    \n{\mathbb{P} \div (v_N \otimes v_N)}_{\dB_{2,q}^{-2}}
    +
    \n{\mathbb{P} \div (w_N \otimes w_N)}_{\dB_{2,q}^{-2}}
    \leq 
    CN^{\frac{1}{q}-1},
\end{align}
which completes the proof.

\noindent
{\bf Acknowledgments.} \\
The author was supported by JSPS KAKENHI, Grant Number JP25K17279.
During the initial brainstorming stage of this research, the author used ChatGPT 5.6 Pro, whereas the author carried out all calculations and assumes full responsibility for all contents.

\noindent
{\bf Conflict of interest.} \\
The author declares no conflict of interest.

\noindent
{\bf Data availability.} \\
No data were generated or analyzed in this study.

\end{document}